\pdfoutput=1
\documentclass{article}

\PassOptionsToPackage{sort}{natbib}
\usepackage[margin=1in]{geometry}
\usepackage{natbib}
\usepackage[utf8]{inputenc} % allow utf-8 input
\usepackage{hyperref}       % hyperlinks
\usepackage{url}            % simple URL typesetting
\usepackage{booktabs}       % professional-quality tables
\usepackage{amsfonts}       % blackboard math symbols
\usepackage{nicefrac}       % compact symbols for 1/2, etc.
\usepackage{microtype}      % microtypography
\usepackage{transparent}

\usepackage{amsmath}
\usepackage{amssymb}
\usepackage{mathtools}
\usepackage{amsthm}
\usepackage{color}
\usepackage[ruled]{algorithm2e}

\usepackage{enumitem}
\usepackage{xspace}

\usepackage{wrapfig}

\newtheorem{lemma}{Lemma}
\newtheorem{theorem}{Theorem}
\newtheorem{definition}{Definition}

\newtheorem{assumption}{Assumption}

\newcommand{\Exp}{\mathrm{Exp}}

\newcommand\numberthis{\addtocounter{equation}{1}\tag{\theequation}}
\newcommand{\ragd}{\textsc{Ragd}}

\renewcommand{\nabla}{\mathrm{grad}}

\title{Towards Riemannian Accelerated Gradient Methods}
\date{}

\usepackage{times}
 \author{
 	Hongyi Zhang \\ BCS \& LIDS, MIT \\ \href{mailto:hongyiz@mit.edu}{hongyiz@mit.edu} \and Suvrit Sra \\ EECS \& LIDS, MIT \\ \href{mailto:suvrit@mit.edu}{suvrit@mit.edu}
 }

\begin{document}

\maketitle

\begin{abstract}
We propose a Riemannian version of Nesterov's Accelerated Gradient algorithm (\ragd), and show that for \emph{geodesically} smooth and strongly convex problems, within a neighborhood of the minimizer whose radius depends on the condition number as well as the sectional curvature of the manifold, \ragd{} converges to the minimizer with acceleration. Unlike the algorithm in \citep{liu2017accelerated} that requires the exact solution to a nonlinear equation which in turn may be intractable, our algorithm is constructive and computationally tractable\footnote{ as long as Riemannian gradient, exponential map and its inverse are computationally tractable, which is the case for many matrix manifolds \citep{absil2009optimization}.}. Our proof exploits a new estimate sequence and a novel bound on the nonlinear metric distortion, both ideas may be of independent interest.
\end{abstract}

\section{Introduction}
%The increasing access to data in virtually every realm has spawned  great interests in analysis and prediction.  , one can reasonably argue that nonlinear optimization algorithms
%Nonlinear optimization is the algorithmic backbone of machine learning where it is instrumental to estimation, classification, regression, reinforcement learning, deep learning, and other key applications. In parallel to its ubiquity in practice, nonlinear optimization has also witnessed great theoretical advances in the past few decades. The most fruitful area for such advances has been convex optimization in vector spaces. Indeed, starting with the seminal ellipsoid algorithm~\citep{khachiyan1980polynomial} a broader complexity theory of convex optimization emerged~\citep{nemirovsky1983problem}. Shortly thereafter, interior point methods~\citep{karmarkar1984new} were discovered, which spurred tremendous growth, culminating in the work~\citep{nesterov1994interior}. 

Convex optimization theory has been a fruitful area of research for decades, with classic work such as the ellipsoid algorithm \citep{khachiyan1980polynomial} and the interior point methods \citep{karmarkar1984new}. However, with the rise of machine learning and data science, growing problem sizes have shifted the community's focus to first-order methods such as gradient descent and stochastic gradient descent.  Over the years, impressive theoretical progress has also been made here, helping elucidate problem characteristics and bringing
insights that drive the discovery of provably faster algorithms, notably Nesterov's accelerated gradient descent~ \citep{nesterov1983method} and variance reduced incremental gradient methods \citep[e.g.,][]{johnson2013accelerating,schmidt2013minimizing,defazio2014saga}.

Outside convex optimization, however, despite some recent progress on nonconvex optimization our theoretical understanding remains limited. Nonetheless, nonconvexity pervades machine learning applications and  motivates identification and study of specialized structure that enables sharper theoretical analysis, e.g., optimality bounds, global complexity, or faster algorithms. Some examples include, problems with low-rank structure~\citep{boumal2016non,ge2017no,sun2017complete,kawaguchi2016deep}; local convergence rates~\citep{ghadimi2013stochastic,Reddi16,Agarwal16,Yair16}; growth conditions that enable fast convergence~\citep{Polyak1963,zhang2016riemannian,attouch2013convergence,shamir2015stochastic}; and nonlinear constraints based on Riemannian manifolds~\citep{boumal2016global,zhang2016first,zhang2016riemannian,mishra2016}, or more general metric spaces~\citep{ambrosio2014metric,bacak2014convex}.

%In practice, nonconvex problems abound in machine learning and data science applications. (Structure assumptions. Introduce Riemannian optimization problems and algorithms.)
%{\color{blue} Landmark result in convex optimization; first-order methods; global iteration complexity}

In this paper, we focus on nonconvexity from a Riemannian viewpoint and consider gradient based optimization. In particular, we are motivated by Nesterov's accelerated gradient method~\citep{nesterov1983method}, a landmark result in the theory of first-order optimization. By introducing an ingenious ``estimate sequence'' technique, \citet{nesterov1983method} devised a first-order algorithm that provably outperforms gradient descent, and is \emph{optimal} (in a first-order oracle model) up to constant factors. This result bridges the gap between the lower and upper complexity bounds in smooth first-order convex optimization~\citep{nemirovsky1983problem, nesterov2004introductory}. 

Following this seminal work, other researchers also developed different analyses to explain the phenomenon of  acceleration. However, both the original proof of Nesterov and all other existing analyses rely heavily on the linear structure of vector spaces. Therefore, our central question is: 
\begin{center}
  \emph{Is linear space structure necessary to achieve acceleration?}
\end{center}
Given that the iteration complexity theory of gradient descent generalizes to Riemannian manifolds~\citep{zhang2016first}, it is tempting to hypothesize that a Riemannian generalization of accelerated gradient methods also works. However, the nonlinear nature of Riemannian geometry poses significant obstructions to either verify or refute such a hypothesis. The aim of this paper is to study existence of accelerated gradient methods on Riemannian manifolds, while identifying and tackling key obstructions and obtaining new tools for global analysis of optimization on Riemannian manifolds as a byproduct.

It is important to note that in a recent work \citep{liu2017accelerated}, the authors claimed to have developed Nesterov-style methods on Riemannian manifolds and analyzed their convergence rates. Unfortunately, this is \emph{not} the case, since their algorithm requires the \emph{exact} solution to a nonlinear equation \cite[(4) and (5)]{liu2017accelerated} on the manifold at every iteration. In fact, solving this nonlinear equation itself can be as difficult as solving the original optimization problem.
%Introduce Riemannian optimization theory.

\subsection{Related work}
The first accelerated gradient method in vector space along with the concept of estimate sequence is proposed by \citet{nesterov1983method}; \citep[Chapter 2.2.1]{nesterov2004introductory} contains an expository introduction. In recent years, there has been a surging interest to either develop new analysis for Nesterov's algorithm or invent new accelerated gradient methods. In particular, \citet{su2014differential,flammarion2015averaging,wibisono2016variational} take a dynamical system viewpoint, modeling the continuous time limit of Nesterov's algorithm as a second-order ordinary differential equation. \citet{allen2014linear} reinterpret Nesterov's algorithm as the linear coupling of a gradient step and a mirror descent step, which also leads to accelerated gradient methods for smoothness defined with non-Euclidean norms. \citet{arjevani2015lower} reinvent Nesterov's algorithm by considering optimal methods for optimizing polynomials. \citet{bubeck2015geometric} develop an alternative accelerated method with a geometric explanation. \citet{lessard2016analysis} use theory from robust control to derive convergence rates for Nesterov's algorithm.

% A collection of accelerated gradient method papers, including \citet{allen2014linear, bubeck2015geometric, flammarion2015averaging, su2014differential, lessard2016analysis, wibisono2016variational, kim2016optimized, arjevani2015lower, ghadimi2016accelerated}.

The design and analysis of Riemannian optimization algorithms as well as some historical perspectives were covered in details in \citep{absil2009optimization}, although the analysis only focused on local convergence. The first global convergence result was derived in \citep{udriste1994convex} under the assumption that the Riemannian Hessian is positive definite. \citet{zhang2016first} established the globally convergence rate of Riemannian gradient descent algorithm for optimizing geodesically convex functions on Riemannian manifolds. Other nonlocal analyses of Riemannian optimization algorithms include stochastic gradient algorithm \citep{zhang2016first}, fast incremental algorithm \citep{zhang2016riemannian, kasai2016riemannian}, proximal point algorithm \citep{ferreira2002proximal} and trust-region algorithm \citep{boumal2016global}. \citet[Chapter 2]{absil2009optimization} also surveyed some important applications of Riemannian optimization.

%{EDIT: \color{blue} Most notably, it generalizes to \emph{geodesically convex} metric spaces~\citep{gromov1978manifolds,bridson1999metric,burago2001course}, through which it offers a much richer setting for developing mathematical models amenable to global optimization. Although Riemannian geometry provides tools that enable generalization of Euclidean algorithms~\citep{udriste1994convex,absil2009optimization}, to obtain iteration complexity bounds we must overcome some fundamental geometric hurdles. We introduce key results that overcome some of these hurdles, and pave the way to analyzing first-order g-convex optimization algorithms.}

%Standard references on Riemannian optimization are \citep{udriste1994convex,absil2009optimization}, but these primarily consider   problems on manifolds without necessarily assuming g-convexity. Consequently, their analysis is limited to asymptotic convergence (except for ~\cite[Theorem 4.2,][]{udriste1994convex} that proves linear convergence for  functions with positive-definite and bounded Riemannian Hessians). The recent monograph \citep{bacak2014convex} is devoted to g-convexity and g-convex optimization on geodesic metric spaces, though without any attention to global complexity analysis. \citet{bacak2014convex} also details a noteworthy application: averaging trees in the geodesic metric space of phylogenetic trees~\citep{billera2001geometry}.

%A collection of Riemannian optimization analysis, including \citet{absil2009optimization,zhang2016first,zhang2016riemannian,kasai2016riemannian,boumal2016global}

\subsection{Summary of results}
In this paper, we make the following contributions:
\begin{enumerate}
	\item We propose the first \emph{computationally tractable} accelerated gradient algorithm that, within a radius from the minimizer that depends on the condition number and sectional curvature bounds, is provably faster than gradient descent methods on Riemannian manifolds with bounded sectional curvatures. (Algorithm \ref{alg:constant-step}, Theorem \ref{thm:convergence-induction})
	\item We analyze the convergence of this algorithm using a new estimate sequence, which relaxes Nesterov's original assumption and also generalizes to Riemannian optimization. (Lemma \ref{thm:estimate-sequence-lemma})
	\item We develop a novel bound related to the bi-Lipschitz property of exponential maps on Riemannian manifolds. This fundamental geometric result is essential for our convergence analysis, but should also have other interesting applications. (Theorem \ref{thm:squared-distance-ratio-bound}) 
\end{enumerate}

%%% Local Variables:
%%% mode: latex
%%% TeX-master: "nips_2017"
%%% End:

\section{Background}

We briefly review concepts in Riemannian geometry that are related to our analysis; for a thorough introduction one standard text is~\citep[e.g.][]{jost2011riemannian}. A \emph{Riemannian manifold} $(\mathcal{M}, \mathfrak{g})$ is a real smooth manifold $\mathcal{M}$ equipped with a Riemannain metric $\mathfrak{g}$. The metric $\mathfrak{g}$ induces an inner product structure on each tangent space $T_x\mathcal{M}$ associated with every $x\in\mathcal{M}$.  We denote the inner product of $u,v\in T_x\mathcal{M}$ as $\langle u, v \rangle \triangleq \mathfrak{g}_x(u,v)$; and the norm of $u\in T_x\mathcal{M}$ is defined as $\|u\|_x \triangleq \sqrt{\mathfrak{g}_x(u,u)}$; we omit the index $x$ for brevity wherever it is obvious from the context. The angle between $u,v$ is defined as $\arccos\frac{\langle u, v \rangle}{\|u\|\|v\|}$. A geodesic is a constant speed curve $\gamma: [0,1]\to\mathcal{M}$ that is locally distance minimizing. An exponential map $\Exp_x:T_x\mathcal{M}\to\mathcal{M}$ maps $v$ in $T_x\mathcal{M}$ to $y$ on $\mathcal{M}$, such that there is a geodesic $\gamma$ with $\gamma(0) = x, \gamma(1) = y$ and $\dot{\gamma}(0) \triangleq \frac{d}{dt}\gamma(0) = v$.  If between any two points in $\mathcal{X}\subset\mathcal{M}$ there is a unique geodesic, the exponential map has an inverse $\Exp_x^{-1}:\mathcal{X}\to T_x\mathcal{M}$ and the geodesic is the unique shortest path with $\|\Exp_x^{-1}(y)\| = \|\Exp_y^{-1}(x)\|$ the geodesic distance between $x,y\in\mathcal{X}$. Parallel transport is the Riemannian analogy of vector translation, induced by the Riemannian metric.

Let $u,v\in T_x\mathcal{M}$ be linearly independent, so that they span a two dimensional subspace of $T_x\mathcal{M}$. Under the exponential map, this subspace is mapped to a two dimensional submanifold of $\mathcal{U}\subset\mathcal{M}$. The sectional curvature $\kappa(x,\mathcal{U})$ is defined as the Gauss curvature of $\mathcal{U}$ at $x$, and is a critical concept in the comparison theorems involving geodesic triangles \citep{burago2001course}.

The notion of geodesically convex sets, geodesically (strongly) convex functions and geodesically smooth functions are defined as straightforward generalizations of the corresponding vector space objects to Riemannian manifolds. In particular,
\begin{itemize}
	\item A set $\mathcal{X}$ is called \emph{geodesically convex} if for any $x,y\in\mathcal{X}$, there is a geodesic $\gamma$ with $\gamma(0) = x, \gamma(1) = y$ and $\gamma(t)\in\mathcal{X}$ for $t\in [0,1]$.
	\item We call a function $f:\mathcal{X}\to\mathbb{R}$ \emph{geodesically convex} (g-convex) if for any $x,y\in\mathcal{X}$ and any geodesic $\gamma$ such that $\gamma(0)=x$, $\gamma(1)=y$ and $\gamma(t)\in\mathcal{X}$ for all $t\in [0,1]$, it holds that
	\[ f(\gamma(t)) \le (1-t)f(x) + tf(y). \]
	It can be shown that if the inverse exponential map is well-defined, an equivalent definition is that for any $x,y\in\mathcal{X}$, $f(y) \ge f(x) + \langle g_x, \Exp_x^{-1}(y) \rangle$,
	where $g_x$ is the gradient of $f$ at $x$ (in this work we assume $f$ is differentiable). A function $f:\mathcal{X}\to\mathbb{R}$ is called \emph{geodesically $\mu$-strongly convex} ($\mu$-strongly g-convex) if for any $x,y\in\mathcal{X}$ and gradient $g_x$, it holds that
	\[ f(y) \ge f(x) + \langle g_x, \Exp_x^{-1}(y) \rangle + \tfrac{\mu}{2}\|\Exp_x^{-1}(y)\|^2.\]
	\item We call a vector field $g :\mathcal{X}\to\mathbb{R}^d$ \emph{geodesically $L$-Lipschitz} ($L$-g-Lipschitz) if for any $x,y\in\mathcal{X}$,
	\[ \|g(x) - \Gamma_y^x g(y)\| \le L \|\Exp_x^{-1}(y)\|, \]
	where $\Gamma_y^x$ is the parallel transport from $y$ to $x$. We call a differentiable function $f:\mathcal{X}\to\mathbb{R}$ \emph{geodesically $L$-smooth} ($L$-g-smooth) if its gradient is $L$-g-Lipschitz, in which case we have
	\[ f(y) \le f(x) + \langle g_x, \Exp_x^{-1}(y) \rangle + \tfrac{L}{2}\|\Exp_x^{-1}(y)\|^2. \]
\end{itemize}
Throughout our analysis, for simplicity, we make the following standing assumptions:
\begin{assumption} \label{assumption:1}
	$\mathcal{X}\subset\mathcal{M}$ is a geodesically convex set where the exponential map $\Exp$ and its inverse $\Exp^{-1}$ are well defined.
\end{assumption}
\begin{assumption} \label{assumption:2}
	The sectional curvature in $\mathcal{X}$ is bounded, i.e. $|\kappa(x,\cdot)|\le K, \forall x\in\mathcal{X}$.
\end{assumption}
\begin{assumption} \label{assumption:3}
	$f$ is geodesically $L$-smooth, $\mu$-strongly convex, and assumes its minimum inside $\mathcal{X}$.
\end{assumption}
\begin{assumption} \label{assumption:4}
	All the iterates remain in $\mathcal{X}$.
\end{assumption}
With these assumptions, the problem being solved can be stated formally as $\min_{x\in\mathcal{X}\subset\mathcal{M}} ~ f(x)$.

%Classic results in Riemannian geometry, esp. Jacobi field estimate, Rauch comparison theorem and its corollary on bi-Lipschitzness of the exponential map \citep[e.g.][Cor. 5.6.1]{jost2011riemannian}. This line of work is closely related to our key lemma, but does not apply to our problem, as we will explain in section 3.

%%% Local Variables:
%%% mode: latex
%%% TeX-master: "colt2018"
%%% End:

\section{Proposed algorithm: \ragd}
\begin{algorithm}[hbtp]
	\caption{Riemannian-Nesterov($x_0, \gamma_0, \{h_k\}_{k=0}^{T-1}, \{\beta_k\}_{k=0}^{T-1}$)} \label{alg:riemannian-ag}
	\SetAlgoLined
	\SetKwInput{KwData}{Parameters}
	\KwData{initial point $x_0\in\mathcal{X}$, $\gamma_0>0$, step sizes $\{h_k\le\frac{1}{L}\}$, shrinkage parameters $\{\beta_k>0\}$}
	initialize $v_0 = x_0$\\
	\For{$k=0,1,\dots,T-1$}{
		Compute $\alpha_k\in(0,1)$ from the equation
		$\alpha_k^2 = h_k\cdot\left((1-\alpha_k)\gamma_k + \alpha_k\mu\right)$\\
		Set $\overline{\gamma}_{k+1} = (1-\alpha_k)\gamma_k + \alpha_k\mu$\\
\nl	\label{ln:y_k}	Choose $y_k = \Exp_{x_k}\left(\frac{\alpha_k\gamma_k}{\gamma_k+\alpha_k\mu}\Exp_{x_k}^{-1}(v_k)\right)$\\
		Compute $f(y_k)$ and $\nabla f(y_k)$\\
\nl	\label{ln:x_k+1}	Set $x_{k+1} = \Exp_{y_k}\left(-h_k\nabla f(y_k)\right)$\label{eq:x-k+1} \\ 
\nl	\label{ln:v_k+1}	Set $v_{k+1} = \Exp_{y_k}\left(\frac{(1-\alpha_k)\gamma_k}{\overline{\gamma}_{k+1}} \Exp_{y_k}^{-1}(v_k) - \frac{\alpha_k}{\overline{\gamma}_{k+1}} \nabla f(y_k)\right)$\label{eq:v-k+1}\\ 
		Set $\gamma_{k+1} = \frac{1}{1+\beta_k}\overline{\gamma}_{k+1}$
	}
	{\bf Output:} $x_T$
\end{algorithm}

\begin{figure}[hbt]
	\centering \def\svgwidth{200pt}
	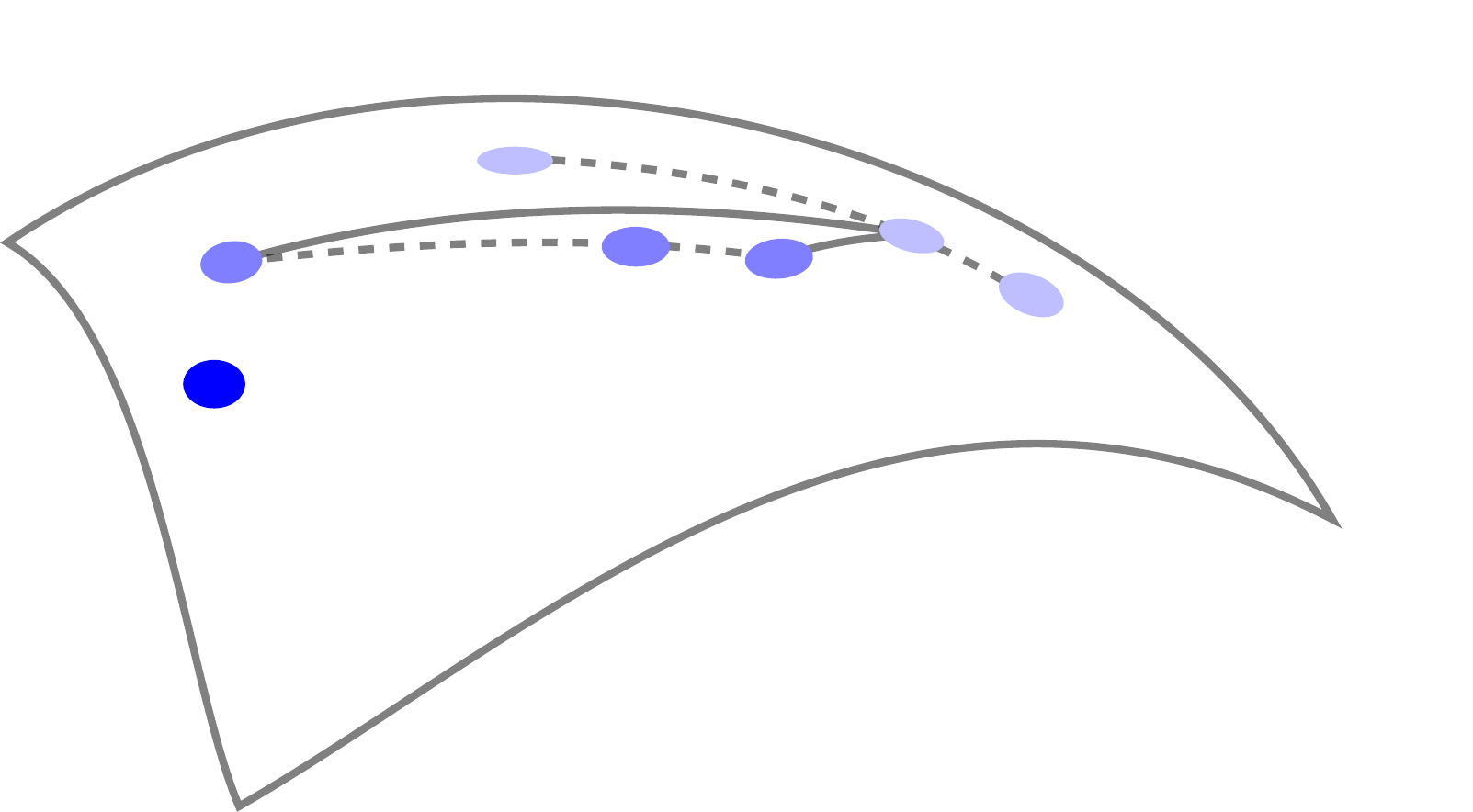 \def\svgwidth{200pt} 
	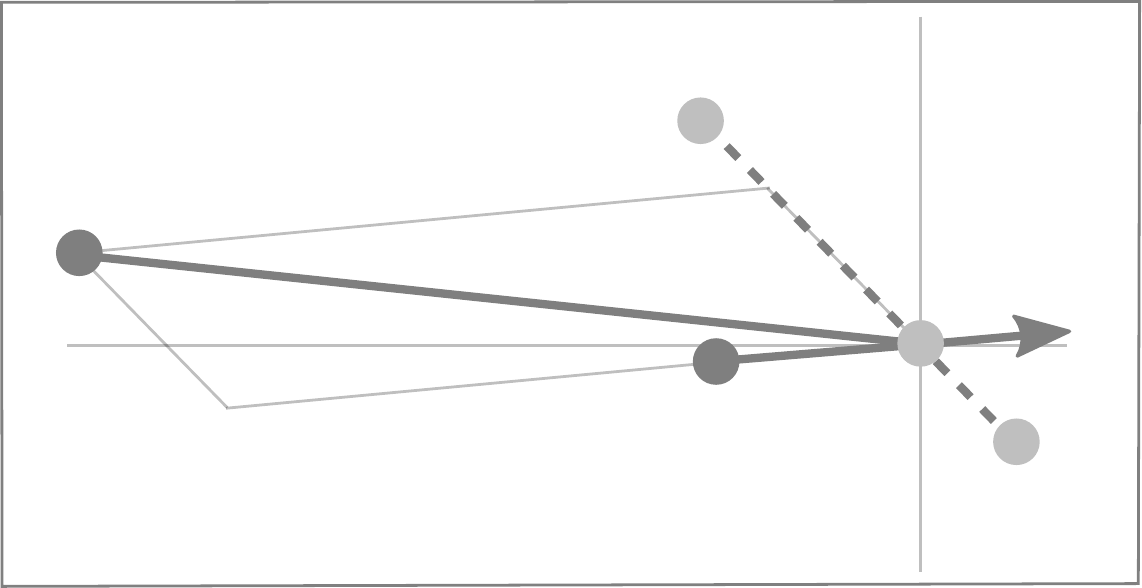
	\caption{Illustration of the geometric quantities in Algorithm \ref{alg:riemannian-ag}. \textbf{Left:} iterates and minimizer $x^*$ with $y_{k}$'s tangent space shown schematically. \textbf{Right:} the inverse exponential maps of relevant iterates in $y_{k}$'s tangent space. Note that $y_k$ is on the geodesic from $x_k$ to $v_k$ (Algorithm \ref{alg:riemannian-ag}, Line \ref{ln:y_k}); $\Exp_{y_k}^{-1}(x_{k+1})$ is in the opposite direction of $\mathrm{grad} f(y_k)$ (Algorithm \ref{alg:riemannian-ag}, Line \ref{ln:x_k+1}); also note how $\Exp_{y_k}^{-1}(v_{k+1})$ is constructed (Algorithm \ref{alg:riemannian-ag}, Line \ref{ln:v_k+1}).}
\end{figure}

Our proposed optimization procedure is shown in Algorithm \ref{alg:riemannian-ag}. We assume the algorithm is granted access to oracles that can efficiently compute the exponential map and its inverse, as well as the Riemannian gradient of function $f$. In comparison with Nesterov's accelerated gradient method in vector space \citep[p.76]{nesterov2004introductory}, we note two important differences: first, instead of linearly combining vectors, the update for iterates is computed via exponential maps; second, we introduce a paired sequence of parameters $\{(\gamma_k, \overline{\gamma}_k)\}_{k=0}^{T-1}$, for reasons that will become clear when we analyze the convergence of the algorithm. 

Algorithm \ref{alg:riemannian-ag} provides a general scheme for Nesterov-style algorithms on Riemannian manifolds, leaving the choice of many parameters to users' preference. To further simplify the parameter choice as well as the analysis, we note that the following specific choice of parameters
\[ \gamma_0\equiv\gamma = \frac{\sqrt{\beta^2+4(1+\beta)\mu h}-\beta}{\sqrt{\beta^2+4(1+\beta)\mu h}+\beta}\cdot \mu, \qquad h_k\equiv h, \forall k\ge 0, \qquad \beta_k\equiv \beta > 0, \forall k\ge 0, \]
which leads to Algorithm \ref{alg:constant-step}, a constant step instantiation of the general scheme. We leave the proof of this claim as a lemma in the Appendix.

\begin{algorithm}[hbtp]
	\caption{Constant Step Riemannian-Nesterov($x_0, h, \beta$)}  \label{alg:constant-step}
	\SetAlgoLined
	\SetKwInput{KwData}{Parameters}
	\KwData{initial point $x_0\in\mathcal{X}$, step size $h\le\frac{1}{L}$, shrinkage parameter $\beta > 0$}
	initialize $v_0 = x_0$\\
	set $\alpha = \frac{\sqrt{\beta^2+4(1+\beta)\mu h}-\beta}{2}$,~ $\gamma = \frac{\sqrt{\beta^2+4(1+\beta)\mu h}-\beta}{\sqrt{\beta^2+4(1+\beta)\mu h}+\beta}\cdot \mu$,~ $\overline{\gamma} = (1+\beta)\gamma$\\
	\For{$k=0,1,\dots,T-1$}{
		Choose $y_k = \Exp_{x_k}\left(\frac{\alpha\gamma}{\gamma+\alpha\mu}\Exp_{x_k}^{-1}(v_k)\right)$\\
		Set $x_{k+1} = \Exp_{y_k}\left(-h\nabla f(y_k)\right)$ \\ 
		Set $v_{k+1} = \Exp_{y_k}\left(\frac{(1-\alpha)\gamma}{\overline{\gamma}} \Exp_{y_k}^{-1}(v_k) - \frac{~\alpha~}{~\overline{\gamma}~} \nabla f(y_k)\right)$
	}
	{\bf Output:} $x_T$
\end{algorithm}

We move forward to analyzing the convergence properties of these two algorithms in the following two sections. In Section \ref{sec:general-analysis}, we first provide a novel generalization of Nesterov's estimate sequence to Riemannian manifolds, then show that if a specific tangent space distance comparison inequality (\ref{eq:base-change-assumption}) always holds, then Algorithm \ref{alg:riemannian-ag} converges similarly as its vector space counterpart. In Section \ref{sec:constant-step-analysis}, we establish sufficient conditions for this tangent space distance comparison inequality to hold, specifically for Algorithm \ref{alg:constant-step}, and show that under these conditions Algorithm \ref{alg:constant-step} converges in $O\left(\sqrt{\frac{L}{\mu}}\log(1/\epsilon)\right)$ iterations, a faster rate than the $O\left(\frac{L}{\mu}\log(1/\epsilon)\right)$ complexity of Riemannian gradient descent.

\section{Analysis of a new estimate sequence} \label{sec:general-analysis}
First introduced in \citep{nesterov1983method}, estimate sequences are central tools in establishing the acceleration of Nesterov's method. We first note a weaker notion of estimate sequences for functions whose domain is not necessarily a vector space.
%\begin{definition}
%	A pair of sequences $\{\Phi_{k}(x):\mathcal{X}\to\mathbb{R}\}_{k=0}^{\infty}$ and $\{\lambda_k\}_{k=0}^{\infty}$ is called an estimate sequence of function $f(x):\mathcal{X}\to\mathbb{R}$ if $\lambda_k\to 0$ and for any $x\in\mathcal{X}$ and all $k\ge 0$ we have:
%	\begin{equation} \label{eq:estimate-sequence-definition}
%	\Phi_k(x) \le (1-\lambda_k)f(x) + \lambda_k\Phi_0(x).
%	\end{equation}
%\end{definition}
\begin{definition} \label{def:weak-estimate-sequence}
	A pair of sequences $\{\Phi_{k}(x):\mathcal{X}\to\mathbb{R}\}_{k=0}^{\infty}$ and $\{\lambda_k\}_{k=0}^{\infty}$ is called a (weak) estimate sequence of a function $f(x):\mathcal{X}\to\mathbb{R}$, if $\lambda_k\to 0$ and for all $k\ge 0$ we have:
	\begin{equation} \label{eq:weak-estimate-sequence-definition}
	\Phi_k(x^*) \le (1-\lambda_k)f(x^*) + \lambda_k\Phi_0(x^*).
	\end{equation}
\end{definition}
This definition relaxes the original definition proposed by \citet[def. 2.2.1]{nesterov2004introductory}, in that the latter requires $\Phi_k(x) \le (1-\lambda_k)f(x) + \lambda_k\Phi_0(x)$ to hold for all $x\in\mathcal{X}$, whereas our definition only assumes it holds at the minimizer $x^*$. We note that similar observations have been made, e.g., in \citep{carmon2017convex}. This relaxation is essential for sparing us from fiddling with the global geometry of Riemannian manifolds.
%
%In the sequel, we will instead use this weak notion of estimate sequence in our analysis.

%Following the above definition, the main theme of this section is to introduce a new type of estimate sequence that generalizes Nesterov's original construction to Riemannian manifolds. 
However, there is one major obstacle in the analysis -- Nesterov's construction of quadratic function sequence critically relies on the linear metric and does not generalize to nonlinear space. An example is given in Figure \ref{fig:change-base}, where we illustrate the distortion of distance (hence quadratic functions) in tangent spaces. The key novelty in our construction is inequality (\ref{eq:phi-less-overline-phi}) which allows a broader family of estimate sequences, as well as inequality (\ref{eq:base-change-assumption}) which handles nonlinear metric distortion and fulfills inequality (\ref{eq:phi-less-overline-phi}). Before delving into the analysis of our specific construction, we recall how to construct estimate sequences and note their use in the following two lemmas.

\begin{lemma} \label{thm:estimate-sequence-construction}
	Let us assume that:
	\begin{enumerate}
          \setlength{\itemsep}{1pt}
		\item $f$ is geodesically $L$-smooth and $\mu$-strongly geodesically convex on domain $\mathcal{X}$.
		\item $\Phi_0(x)$ is an arbitrary function on $\mathcal{X}$.
		\item $\{y_k\}_{k=0}^{\infty}$ is an arbitrary sequence in $\mathcal{X}$.
		\item $\{\alpha_k\}_{k=0}^{\infty}$: $\alpha_k\in(0,1)$,  $\sum_{k=0}^{\infty} \alpha_k = \infty$. \label{eq:alpha-k-not-summable}
		\item $\lambda_0 = 1$.
	\end{enumerate}
	Then the pair of sequences $\{\Phi_k(x)\}_{k=0}^{\infty}$, $\{\lambda_k\}_{k=0}^{\infty}$ which satisfy the following recursive rules:
	\begin{align}
	\lambda_{k+1} = &~ (1-\alpha_k)\lambda_k, \\
	\overline{\Phi}_{k+1}(x) = &~ (1-\alpha_k)\Phi_k(x) + \alpha_k \left[f(y_k) + \langle \nabla f(y_k), \Exp_{y_k}^{-1}(x)\rangle + \frac{\mu}{2}\|\Exp_{y_k}^{-1}(x)\|^2\right], \label{eq:phi-recursion}\\
	\Phi_{k+1}(x^*) \le &~ \overline{\Phi}_{k+1}(x^*), \label{eq:phi-less-overline-phi}
	\end{align}
	is a (weak) estimate sequence.
\end{lemma}
	The proof is similar to~\citep[Lemma 2.2.2]{nesterov2004introductory} which we include in Appendix \ref{prf:estimate-sequence-construction}.

\begin{lemma} \label{thm:estimate-sequence-implication}
	If for a (weak) estimate sequence $\{\Phi_{k}(x):\mathcal{X}\to\mathbb{R}\}_{k=0}^{\infty}$ and $\{\lambda_k\}_{k=0}^{\infty}$ we can find a sequence of iterates $\{x_k\}$, such that
	\[ f(x_k) \le \Phi_k^* \equiv \min_{x\in\mathcal{X}}\Phi_k(x), \]
	then $f(x_k)-f(x^*) \le \lambda_k(\Phi_0(x^*)-f(x^*)) \to 0$.
\end{lemma}
\begin{proof} By Definition \ref{def:weak-estimate-sequence} we have
	$f(x_k)\le \Phi_k^* \le \Phi_k(x^*) \le (1-\lambda_k)f(x^*) + \lambda_k\Phi_0(x^*)$.
	Hence $f(x_k) - f(x^*) \le \lambda_k(\Phi_0(x^*) - f(x^*)) \to 0$.
\end{proof}
Lemma \ref{thm:estimate-sequence-implication} immediately suggest the use of (weak) estimate sequences in establishing the convergence and analyzing the convergence rate of certain iterative algorithms. The following lemma shows that a weak estimate sequence exists for Algorithm \ref{alg:riemannian-ag}. Later in Lemma \ref{thm:x-k-bound}, we prove that the sequence $\{x_k\}$ in Algorithm \ref{alg:riemannian-ag} satisfies the requirements in Lemma \ref{thm:estimate-sequence-implication} for our estimate sequence.

\begin{lemma} \label{thm:estimate-sequence-lemma}
	Let $\Phi_0(x) = \Phi_0^* + \frac{\gamma_0}{2}\|\Exp_{y_0}^{-1}(x)\|^2$.
	Assume for all $k\ge 0$, the sequences $\{\gamma_k\}$, $\{\overline{\gamma}_k\}$, $\{v_k\}$, $\{\Phi_k^*\}$ and $\{\alpha_k\}$ satisfy
	\begin{align}
	\overline{\gamma}_{k+1} = &~ (1-\alpha_k)\gamma_k + \alpha_k\mu, \label{eq:overline-gamma-k+1}\\
	v_{k+1} = &~ \Exp_{y_k}\left(\frac{(1-\alpha_k)\gamma_k}{\overline{\gamma}_{k+1}} \Exp_{y_k}^{-1}(v_k) - \frac{\alpha_k}{\overline{\gamma}_{k+1}} \nabla f(y_k)\right)  \\
	\nonumber \Phi_{k+1}^* = &~ \left(1 - \alpha_k\right) \Phi_k^* + \alpha_k f(y_k) - \frac{\alpha_k^2}{2\overline{\gamma}_{k+1}}\|\nabla f(y_k)\|^2 \\
	&~ + \frac{\alpha_k(1-\alpha_k)\gamma_k}{\overline{\gamma}_{k+1}}\left(\frac{\mu}{2}\|\Exp_{y_k}^{-1}(v_k)\|^2 + \langle \nabla f(y_k), \Exp_{y_k}^{-1}(v_k)\rangle \right), \label{eq:phi-k+1-star} \\
	\gamma_{k+1} \| \Exp&_{y_{k+1}}^{-1}(x^*) -\Exp_{y_{k+1}}^{-1}(v_{k+1})\|^2 \le \overline{\gamma}_{k+1}\|\Exp_{y_k}^{-1}(x^*)-\Exp_{y_k}^{-1}(v_{k+1})\|^2, \label{eq:base-change-assumption} \\
	\alpha_k\in &~ (0,1), \quad \sum_{k=0}^{\infty} \alpha_k = \infty,
	\end{align}
	then the pair of sequence $\{\Phi_k(x)\}_{k=0}^{\infty}$ and $\{\lambda_k\}_{k=0}^{\infty}$, defined by
	\begin{align}
	\Phi_{k+1}(x) = &~ \Phi_{k+1}^* + \frac{\gamma_{k+1}}{2}\|\Exp_{y_{k+1}}^{-1}(x)-\Exp_{y_{k+1}}^{-1}(v_{k+1})\|^2, \\
	\lambda_0 = 1,  \quad &~ \lambda_{k+1} = (1-\alpha_k)\lambda_k.
	\end{align}
	is a (weak) estimate sequence.
\end{lemma}
\begin{proof} Recall the definition of $\overline{\Phi}_{k+1}(x)$ in Equation (\ref{eq:phi-recursion}). We claim that if $\Phi_k(x) = \Phi_k^* + \frac{\gamma_k}{2}\|\Exp_{y_k}^{-1}(x)-\Exp_{y_k}^{-1}(v_k)\|^2$, then we have $\overline{\Phi}_{k+1}(x) \equiv \Phi_{k+1}^* + \frac{\overline{\gamma}_{k+1}}{2}\|\Exp_{y_k}^{-1}(x)-\Exp_{y_k}^{-1}(v_{k+1})\|^2$. The proof of this claim requires a simple algebraic manipulation as is noted as Lemma \ref{thm:complete-square}. Now using the assumption (\ref{eq:base-change-assumption}) we immediately get $\Phi_{k+1}(x^*)\le\overline{\Phi}_{k+1}(x^*)$. By Lemma \ref{thm:estimate-sequence-construction} the proof is complete.
\end{proof}
We verify the specific form of $\overline{\Phi}_{k+1}(x)$ in Lemma~\ref{thm:complete-square}, whose proof can be found in the Appendix \ref{prf:complete-square}.
\begin{lemma} \label{thm:complete-square}
	For all $k\ge 0$, if $\Phi_k(x) = \Phi_k^* + \frac{\gamma_k}{2}\|\Exp_{y_k}^{-1}(x)-\Exp_{y_k}^{-1}(v_k)\|^2$, then with $\overline{\Phi}_{k+1}$ defined as in (\ref{eq:phi-recursion}), $\overline{\gamma}_{k+1}$ as in (\ref{eq:overline-gamma-k+1}), $v_{k+1}$ as in Algorithm \ref{alg:riemannian-ag} and $\Phi_{k+1}^*$ as in ($\ref{eq:phi-k+1-star}$) we have $\overline{\Phi}_{k+1}(x) \equiv \Phi_{k+1}^* + \frac{\overline{\gamma}_{k+1}}{2}\|\Exp_{y_k}^{-1}(x)-\Exp_{y_k}^{-1}(v_{k+1})\|^2$.
\end{lemma}
The next lemma asserts that the iterates $\{x_k\}$ of Algorithm \ref{alg:riemannian-ag} satisfy the requirement that the function values $f(x_k)$ are upper bounded by $\Phi_k^*$ defined in our estimate sequence.
\begin{lemma} \label{thm:x-k-bound}
	Assume $\Phi_0^*=f(x_0)$, and $\{\Phi_k^*\}$ be defined as in (\ref{eq:phi-k+1-star}) with $\{x_k\}$ and other terms defined as in Algorithm \ref{alg:riemannian-ag}. Then we have $\Phi_{k}^*\ge f(x_k)$ for all $k\ge 0$.
\end{lemma}
The proof is standard. We include it in Appendix \ref{prf:x-k-bound} for completeness.
Finally, we are ready to state the following theorem on the convergence rate of Algorithm \ref{alg:riemannian-ag}.

\begin{theorem}[Convergence of Algorithm \ref{alg:riemannian-ag}] \label{thm:main-theorem-general-scheme}
	For any given $T\ge 0$, assume (\ref{eq:base-change-assumption}) is satisfied for all $0\le k\le T$, then Algorithm \ref{alg:riemannian-ag} generates a sequence $\{x_k\}_{k=0}^{\infty}$ such that
	\begin{equation} \label{eq:convergence-algorithm-1}
	 f(x_T) - f(x^*) \le \lambda_T\left(f(x_0) - f(x^*) + \frac{\gamma_0}{2}\|\Exp_{x_0}^{-1}(x^*)\|^2 \right)
	 \end{equation}
	where $\lambda_0 = 1$ and $\lambda_k = \prod_{i=0}^{k-1}(1-\alpha_i)$.
\end{theorem}
\begin{proof}
	The proof is similar to \citep[Theorem 2.2.1]{nesterov2004introductory}. We choose $\Phi_0(x) = f(x_0) + \frac{\gamma_0}{2}\|\Exp_{y_0}^{-1}(x)\|^2$, hence $\Phi_0^* = f(x_0)$. By Lemma \ref{thm:estimate-sequence-lemma} and Lemma \ref{thm:x-k-bound}, the assumptions in Lemma \ref{thm:estimate-sequence-implication} hold. It remains to use Lemma \ref{thm:estimate-sequence-implication}.
\end{proof}

\section{Local fast rate with a constant step scheme} \label{sec:constant-step-analysis}
% Towards this goal, we first make a few specification to our general scheme in Algorithm \ref{alg:riemannian-ag} to simplify the choice of parameters.
%\begin{lemma}[Bounding $\|\nabla f(y_k)\|$]
%	\begin{equation}
%	\|\nabla f(y_k)\|^2 \le 2L (f(y_k)-f(x^*))
%	\end{equation}
%\end{lemma}
%\begin{proof}
%	Use $L$-smooth assumption and note that $x^*$ is the minimizer, we have
%	\[ f(x^*)\le f\left(\Exp_{y_k}\left(-\nabla f(y_k)/L\right)\right) \le f(y_k) - \frac{1}{2L}\|\nabla f(y_k)\|^2. \]
%	Rearrange the terms and the proof is complete.
%\end{proof}

By now we see that almost all the analysis of Nesterov's generalizes, except that the assumption in (\ref{eq:base-change-assumption}) is not necessarily satisfied. In vector space, the two expressions both reduce to $x^* - v_{k+1}$ and hence (\ref{eq:base-change-assumption}) trivially holds with $\gamma = \overline{\gamma}$. On Riemannian manifolds, however, due to the nonlinear Riemannian metric and the associated exponential maps,  $\| \Exp_{y_{k+1}}^{-1}(x^*) -\Exp_{y_{k+1}}^{-1}(v_{k+1})\|$ and $\|\Exp_{y_k}^{-1}(x^*)-\Exp_{y_k}^{-1}(v_{k+1})\|$ in general do not equal (illustrated in Figure \ref{fig:change-base}). Bounding the difference between these two quantities points the way forward for our analysis, which is also our main contribution in this section. We start with two lemmas comparing a geodesic triangle and the triangle formed by the preimage of its vertices in the tangent space, in two constant curvature spaces: hyperbolic space and the hypersphere.
\begin{figure}[hbt]
	\centering \hspace{30pt} \def\svgwidth{220pt}
	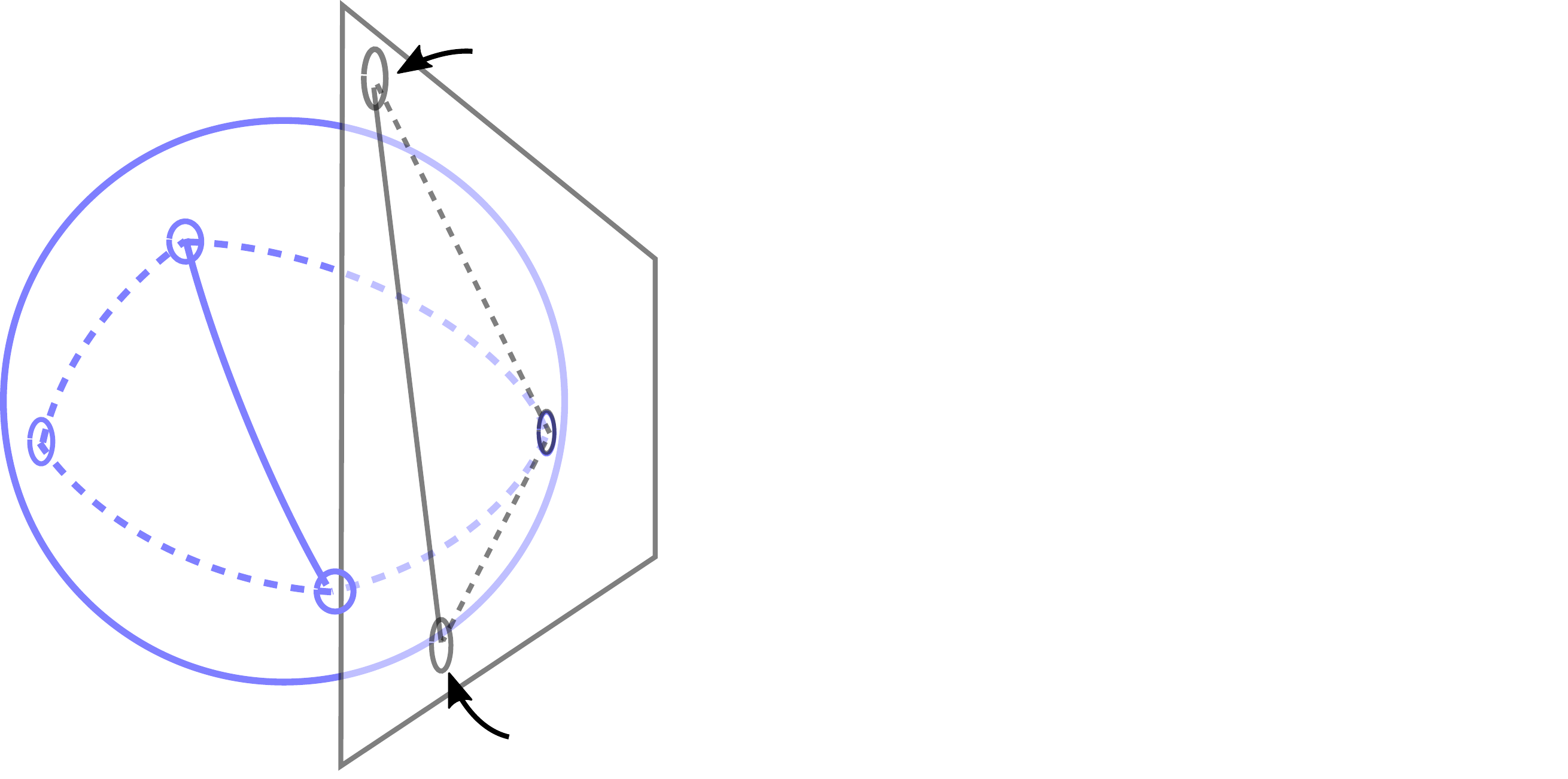 \hspace{-80pt} \def\svgwidth{190pt} 
	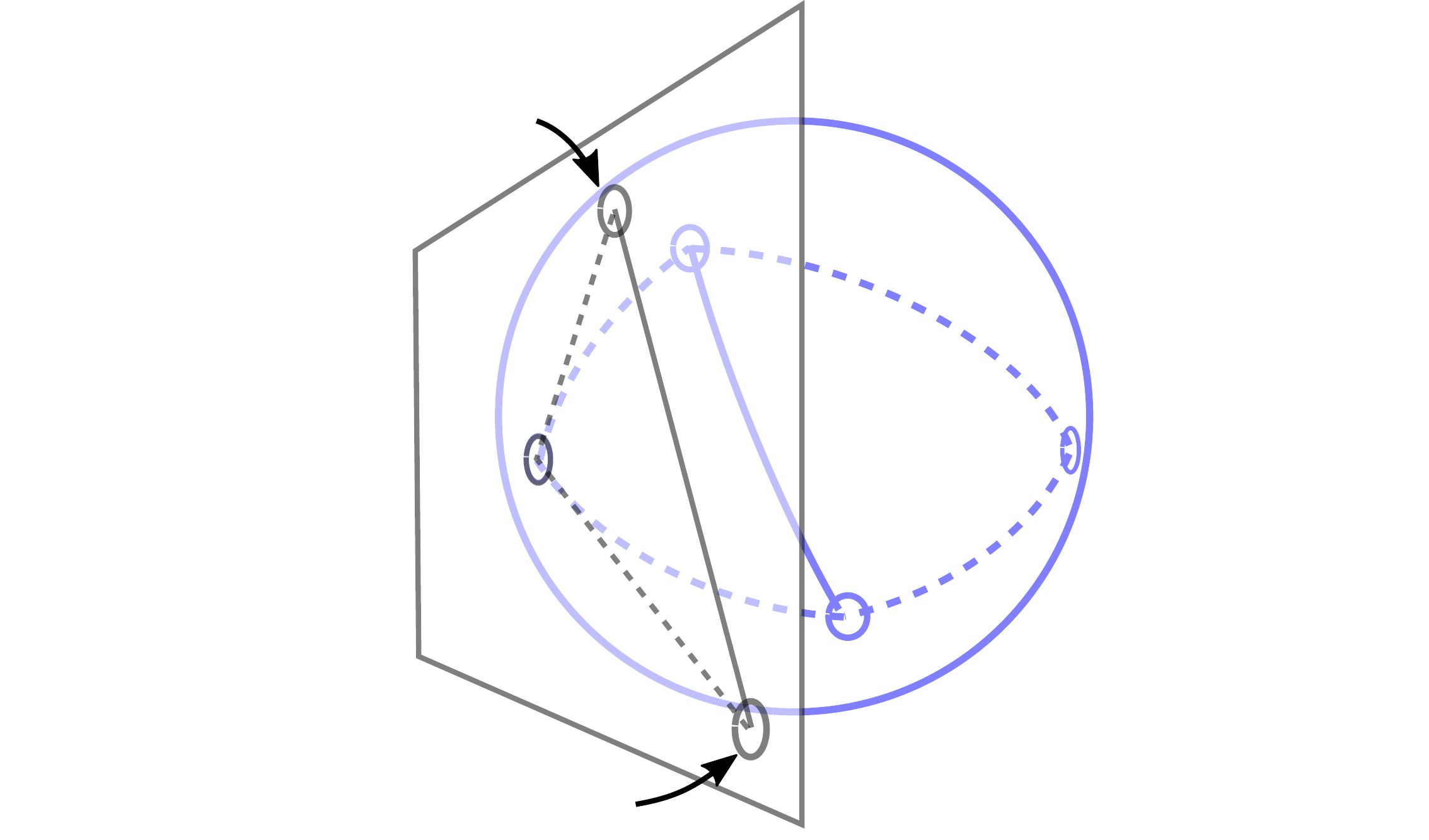 \hspace{-30pt}
	\caption{A schematic illustration of the geometric quantities in Theorem \ref{thm:squared-distance-ratio-bound}. Tangent spaces of $y_{k}$ and $y_{k+1}$ are shown in separate figures to reduce cluttering. Note that even on a sphere (which has constant positive sectional curvature), $d(x^*, v_{k+1}), \|\Exp_{y_{k}}^{-1}(x^*)-\Exp_{y_{k}}^{-1}(v_{k+1})\|$ and $ \|\Exp_{y_{k+1}}^{-1}(x^*)-\Exp_{y_{k+1}}^{-1}(v_{k+1})\|$ generally do not equal.} \label{fig:change-base}
\end{figure}

\begin{lemma}[bi-Lipschitzness of the exponential map in hyperbolic space] \label{thm:hyperbolic-squared-distance-distortion}
	Let $a,b,c$ be the side lengths of a geodesic triangle in a hyperbolic space with constant sectional curvature $-1$, and $A$ is the angle between sides $b$ and $c$. Furthermore, assume $b\le\frac{1}{4},c\ge 0$. Let $\triangle\bar{a}\bar{b}\bar{c}$ be the comparison triangle in Euclidean space, with $\bar{b}=b,\bar{c}=c,\bar{A}=A$, then
	\begin{equation}
	\bar{a}^2 \le a^2\le (1+2b^2)\bar{a}^2.
	\end{equation}
\end{lemma}
\begin{proof}
	The proof of this lemma contains technical details that deviate from our main focus; so we defer them to the appendix. The first inequality is well known. To show the second inequality, we have Lemma \ref{thm:large-c-hyperbolic} and Lemma \ref{thm:small-c-hyperbolic} (in Appendix) which in combination complete the proof.
\end{proof}
We also state without proof that by the same techniques one can show the following result holds.
\begin{lemma}[bi-Lipschitzness of the exponential map on hypersphere] \label{thm:hypersphere-squared-distance-distortion}
	Let $a,b,c$ be the side lengths of a geodesic triangle in a hypersphere with constant sectional curvature $1$, and $A$ is the angle between sides $b$ and $c$. Furthermore, assume $b\le\frac{1}{4},c\in[0,\frac{\pi}{2}]$. Let $\triangle\bar{a}\bar{b}\bar{c}$ be the comparison triangle in Euclidean space, with $\bar{b}=b,\bar{c}=c,\bar{A}=A$, then
	\begin{equation}
	a^2\le \bar{a}^2\le (1+2b^2)a^2.
	\end{equation}
\end{lemma}
Albeit very much simplified, spaces of constant curvature are important objects to study, because often their properties can be generalized to general Riemannian manifolds with bounded curvature, specifically via the use of powerful comparison theorems in metric geometry \citep{burago2001course}. In our case, we use these two lemmas to derive a tangent space distance comparison theorem for Riemannian manifolds with bounded sectional curvature.
\begin{theorem}[Multiplicative distortion of squared distance on Riemannian manifold] \label{thm:squared-distance-ratio-bound}
	Let $x^*$, $v_{k+1}$, $y_k$, $y_{k+1}\in\mathcal{X}$ be four points in a g-convex, uniquely geodesic set $\mathcal{X}$ where the sectional curvature is bounded within $[-K, K]$, for some nonnegative number $K$.
%	 where $\kappa_{\max}>0>\kappa_{\min}$ and $K = \max\{-\kappa_{\min},\kappa_{\max}\}$, 
	Define $b_{k+1}=\max\left\{\|\Exp_{y_k}^{-1}(x^*)\|,\|\Exp_{y_{k+1}}^{-1}(x^*)\|\right\}$. Assume $b_{k+1}\le\frac{1}{4\sqrt{K}}$ for $K>0$ (otherwise $b_{k+1} < \infty$), then we have 
	\begin{equation}
	\|\Exp_{y_{k+1}}^{-1}(x^*)-\Exp_{y_{k+1}}^{-1}(v_{k+1})\|^2 \le (1+5K b_{k+1}^2) \|\Exp_{y_k}^{-1}(x^*)-\Exp_{y_k}^{-1}(v_{k+1})\|^2.
	\end{equation}
\end{theorem}

\begin{proof}
	The high level idea is to think of the tangent space distance distortion on Riemannian manifolds of bounded curvature as a consequence of bi-Lipschitzness of the exponential map. Specifically, note that $\triangle y_k x^* v_{k+1}$ and $\triangle y_{k+1} x^* v_{k+1}$ are two geodesic triangles in $\mathcal{X}$, whereas $\|\Exp_{y_k}^{-1}(x^*)-\Exp_{y_k}^{-1}(v_{k+1})\|$ and $\|\Exp_{y_{k+1}}^{-1}(x^*)-\Exp_{y_{k+1}}^{-1}(v_{k+1})\|$ are side lengths of two comparison triangles in vector space. Since $\mathcal{X}$ is of bounded sectional curvature, we can apply comparison theorems.
	
	First, we consider bound on the distortion of squared distance in a Riemannian manifold with constant curvature $-K$. Note that in this case, the hyperbolic law of cosines becomes
	\[ \cosh(\sqrt{K}a) = \cosh(\sqrt{K}b)\cosh(\sqrt{K}c) - \sinh(\sqrt{K}b)\sinh(\sqrt{K}c)\cos(A), \]
	which corresponds to the geodesic triangle in hyperbolic space with side lengths $\sqrt{K}a, \sqrt{K}b,\sqrt{K}c$, with the corresponding comparison triangle in Euclidean space having lengths $\sqrt{K}\bar{a}, \sqrt{K}\bar{b}, \sqrt{K}\bar{c}$. Apply Lemma \ref{thm:hyperbolic-squared-distance-distortion} we have $(\sqrt{K}a)^2 \le (1+2(\sqrt{K}b)^2)(\sqrt{K}\bar{a})^2$, i.e. $a^2 \le (1+2K b^2)\bar{a}^2$.
	%	and the squared distance ratio becomes
	%	\begin{align}
	%	\nonumber &~ \frac{\left(\arccosh\left(\cosh(\sqrt{K}b)\cosh(\sqrt{K}c) - \sinh(\sqrt{K}b)\sinh(\sqrt{K}c)\cos(A)\right)/\sqrt{K}\right)^2}{b^2 + c^2 - 2bc\cos(A)} \\
	%	\nonumber = &~ \frac{\left(\arccosh\left(\cosh(\sqrt{K}b)\cosh(\sqrt{K}c) - \sinh(\sqrt{K}b)\sinh(\sqrt{K}c)\cos(A)\right)\right)^2}{(\sqrt{K}b)^2 + (\sqrt{K}c)^2 - 2(\sqrt{K}b)(\sqrt{K}c)\cos(A)} \\
	%	\le &~ 1 + 2(\sqrt{K}b)^2 = 1 + 2K b^2 \label{eq:hyperbolic-squared-distance-ratio}
	%	\end{align}
	%	assuming $\sqrt{K}b\le\frac{1}{4}$, where the inequality is by using lengths $\sqrt{K}b, \sqrt{K}c$ in Lemma \ref{thm:hyperbolic-squared-distance-distortion}.
	Now consider the geodesic triangle $\triangle y_k x^* v_{k+1}$. Let $\tilde{a}=\|\Exp_{v_{k+1}}^{-1}(x^*)\|, b=\|\Exp_{y_k}^{-1}(v_{k+1})\|\le b_{k+1}, c=\|\Exp_{y_k}^{-1}(x^*)\|, A=\angle x^* y_k v_{k+1}$, so that $\|\Exp_{y_k}^{-1}(x^*)-\Exp_{y_k}^{-1}(v_{k+1})\|^2 = b^2+c^2-2bc\cos(A)$. By Toponogov's comparison theorem \citep{burago2001course}, we have $\tilde{a}\le a$
	hence
	\begin{equation} \label{eq:y-k-squared-distance-ratio}
	\|\Exp_{v_{k+1}}^{-1}(x^*)\|^2 \le \left(1+2K b_{k+1}^2\right)\|\Exp_{y_k}^{-1}(x^*)-\Exp_{y_k}^{-1}(v_{k+1})\|^2.
	\end{equation}
	Similarly, using the spherical law of cosines for a space of constant curvature $K$ 
	\[ \cos(\sqrt{K}a) = \cos(\sqrt{K}b)\cos(\sqrt{K}c) + \sin(\sqrt{K}b)\sin(\sqrt{K}c)\cos(A) \]
	and Lemma \ref{thm:hypersphere-squared-distance-distortion} we can show $\bar{a}^2 \le (1+2K b^2)a^2$, where $\bar{a}$ is the  side length in Euclidean space corresponding to $a$. 
	%	corresponds to the geodesic triangle in hypersphere with side lengths $\sqrt{K}a, \sqrt{K}b,\sqrt{K}c$. For 
	%	and lengths $\sqrt{K}b, \sqrt{K}c$ in Lemma \ref{thm:hypersphere-squared-distance-distortion}, we have that when comparing geodesic triangles $(\{b,c\},A)$ in Euclidean space and a Riemannian manifold with constant curvature $K$, assuming $\sqrt{K}b\le\frac{1}{4}, \sqrt{K}c\le\frac{\pi}{2}$, the squared distance ratio is bounded by
	%	\begin{align}
	%	\nonumber &~ \frac{b^2 + c^2 - 2bc\cos(A)}{\left(\arccos\left(\cos(\sqrt{K}b)\cos(\sqrt{K}c) + \sin(\sqrt{K}b)\sin(\sqrt{K}c)\cos(A)\right)/\sqrt{K}\right)^2} \\
	%	\nonumber = &~ \frac{(\sqrt{K}b)^2 + (\sqrt{K}c)^2 - 2(\sqrt{K}b)(\sqrt{K}c)\cos(A)}{\left(\arccos\left(\cos(\sqrt{K}b)\cos(\sqrt{K}c) + \sin(\sqrt{K}b)\sin(\sqrt{K}c)\cos(A)\right)\right)^2} \\
	%	\le &~ 1 + 2(\sqrt{K}b)^2 = 1 + 2K b^2 \label{eq:hypersphere-squared-distance-ratio}
	%	\end{align}
	%	
	%	Similarly, for the geodesic triangle $\triangle y_{k+1} x v_{k+1}$. We now let $a=\|\Exp_{v_{k+1}}^{-1}(x)\|, b=\|\Exp_{y_{k+1}}^{-1}(v_{k+1})\|, c=\|\Exp_{y_{k+1}}^{-1}(x)\|, A=\angle x y_{k+1} v_{k+1}$, so that $\|\Exp_{y_{k+1}}^{-1}(x)-\Exp_{y_{k+1}}^{-1}(v_{k+1})\|^2 = b^2+c^2-2bc\cos(A)$. Using a corollary of the Rauch comparison theorem (TODO: note this is a local result; assumptions required), we have that
	Hence by our uniquely geodesic assumption and \citep[Theorem 2.2, Remark 7]{meyer1989toponogov}, with similar reasoning for the geodesic triangle $\triangle y_{k+1} x^* v_{k+1}$, we have $a \le \|\Exp_{v_{k+1}}^{-1}(x^*)\|$, so that
	\begin{equation} \label{eq:y-k+1-squared-distance-ratio}
	\|\Exp_{y_{k+1}}^{-1}(x^*)-\Exp_{y_{k+1}}^{-1}(v_{k+1})\|^2 \le \left(1+2K b_{k+1}^2\right)a^2 \le \left(1+2K b_{k+1}^2\right)\|\Exp_{v_{k+1}}^{-1}(x^*)\|^2.
	\end{equation}
	Finally, combining inequalities (\ref{eq:y-k-squared-distance-ratio}) and (\ref{eq:y-k+1-squared-distance-ratio}), and noting that $(1+2K b_{k+1}^2)^2 = 1+4K b_{k+1}^2 + (4K b_{k+1}^2)K b^2 \le 1 + 5K b_{k+1}^2$, the proof is complete.
\end{proof}
Theorem \ref{thm:squared-distance-ratio-bound} suggests that if $b_{k+1}\le\frac{1}{4\sqrt{K}}$, we could choose $\beta\ge 5K b_{k+1}^2$ and $\gamma \le\frac{1}{1+\beta}\overline{\gamma}$ to guarantee $\Phi_{k+1}(x^*)\le \overline{\Phi}_{k+1}(x^*)$. It then follows that the analysis holds for $k$-th step. Still, it is unknown that under what conditions can we guarantee $\Phi_{k+1}(x^*)\le \overline{\Phi}_{k+1}(x^*)$ hold for all $k\ge 0$, which would lead to a convergence proof. We resolve this question in the next theorem. 

\begin{theorem}[Local fast convergence] \label{thm:convergence-induction}
	With Assumptions \ref{assumption:1}, \ref{assumption:2}, \ref{assumption:3}, \ref{assumption:4}, denote $D = \frac{1}{20\sqrt{K}}\left(\frac{\mu}{L}\right)^{\frac{3}{4}}$ and assume $\mathcal{B}_{x^*, D}:=\{x\in\mathcal{M}: d(x,x^*)\le D\} \subseteq\mathcal{X}$.
	 If we set $h=\frac{1}{L}, \beta=\frac{1}{5}\sqrt{\frac{\mu}{L}}$ and $x_0 \in \mathcal{B}_{x^*, D}$,
	then Algorithm \ref{alg:constant-step} converges; moreover, we have
	\begin{equation} \label{eq:convergence-rate}
	f(x_k)-f(x^*)\le \left(1-\frac{9}{10}\sqrt{\frac{\mu}{L}}\right)^k\left(f(x_0)-f(x^*)+\frac{\mu}{2}\|\Exp_{x_0}^{-1}(x^*)\|^2\right).
	\end{equation}
\end{theorem}

\paragraph{Proof sketch.}
Recall that in Theorem 		\ref{thm:main-theorem-general-scheme} we already establish that if the tangent space distance comparison inequality (\ref{eq:base-change-assumption}) holds, then the general Riemannian Nesterov iteration (Algorithm \ref{alg:riemannian-ag}) and hence its constant step size special case (Algorithm \ref{alg:constant-step}) converge with a guaranteed rate. By the tangent space distance comparison theorem (Theorem \ref{thm:squared-distance-ratio-bound}), the comparison inequality should hold if $y_k$ and $x^*$ are close enough. Indeed, 
we use induction to assert that with a good initialization, (\ref{eq:base-change-assumption}) holds for each step. Specifically, for every $k>0$, if $y_k$ is close to $x^*$ and the comparison inequality holds until the $(k-1)$-th step, then $y_{k+1}$ is also close to $x^*$ and the comparison inequality holds until the $k$-th step. We postpone the complete proof until Appendix \ref{prf:convergence-induction}.

\section{Discussion}
In this work, we proposed a Riemannian generalization of the accelerated gradient algorithm and developed its convergence and complexity analysis. For the first time (to the best of our knowledge), we show gradient based algorithms on Riemannian manifolds can be accelerated, at least in a neighborhood of the minimizer. Central to our analysis are the two main technical contributions of our work: a new estimate sequence (Lemma \ref{thm:estimate-sequence-lemma}), which relaxes the assumption of Nesterov's original construction and handles metric distortion on Riemannian manifolds; a tangent space distance comparison theorem (Theorem \ref{thm:squared-distance-ratio-bound}), which provides sufficient conditions for bounding the metric distortion and could be of interest for a broader range of problems on Riemannian manifolds.

Despite not matching the standard convex results, our result exposes the key difficulty of analyzing Nesterov-style algorithms on Riemannian manifolds, an aspect missing in previous work. Critically, the convergence analysis relies on bounding a new distortion term per each step. Furthermore, we observe that the side length sequence $d(y_k, v_{k+1})$ can grow much greater than $d(y_k, x^*)$, even if we reduce the ``step size'' $h_k$ in Algorithm 1, defeating any attempt to control the distortion globally by modifying the algorithm parameters. This is a benign feature in vector space analysis, since (\ref{eq:base-change-assumption}) trivially holds nonetheless; however it poses a great difficulty for analysis in nonlinear space. Note the stark contrast to (stochastic) gradient descent, where the step length can be effectively controlled by reducing the step size, hence bounding the distortion terms globally \citep{zhang2016first}.

A topic of future interest is to study whether assumption (\ref{eq:base-change-assumption}) can be further relaxed, while maintaining that overall the algorithm still converges. By bounding the squared distance distortion in every step, our analysis provides guarantee for the worst-case scenario, which seems unlikely to happen in practice. It would be interesting to conduct experiments to see how often (\ref{eq:base-change-assumption}) is violated versus how often it is loose. It would also be interesting to construct some adversarial problem case (if any) and study the complexity lower bound of gradient based Riemannian optimization, to see if geodesically convex optimization is strictly more difficult than convex optimization. Generalizing the current analysis to non-strongly g-convex functions is another interesting direction.

\section*{Acknowledgement}

The authors thank the anonymous reviewers for helpful feedback. This work was supported in part by NSF-IIS-1409802 and the DARPA Lagrange grant.

\bibliographystyle{abbrvnat}
\bibliography{main_agd17}

\newpage
\appendix

\section{Constant step scheme}
\begin{lemma}\label{thm:constant-step-scheme}
	Pick $\beta_k\equiv \beta > 0$. 
	If in Algorithm \ref{alg:riemannian-ag} we set \[ h_k\equiv h, \forall k\ge 0, \qquad \gamma_0\equiv\gamma = \frac{\sqrt{\beta^2+4(1+\beta)\mu h}-\beta}{\sqrt{\beta^2+4(1+\beta)\mu h}+\beta}\cdot \mu, \]
	then we have 
	\begin{equation}
	\alpha_k\equiv \alpha = \frac{\sqrt{\beta^2+4(1+\beta)\mu h}-\beta}{2}, \qquad \overline{\gamma}_{k+1}\equiv (1+\beta)\gamma, \qquad \gamma_{k+1}\equiv \gamma, \qquad \forall k\ge 0.
	\end{equation}
\end{lemma}
\begin{proof}
	Suppose that $\gamma_k=\gamma$, then from Algorithm \ref{alg:riemannian-ag} we have $\alpha_k$ is the positive root of
	\[ \alpha_k^2-(\mu-\gamma)h\alpha_k - \gamma h = 0. \]
	Also note
	\begin{equation} \label{eq:mu-and-gamma-0}
	\mu - \gamma = \frac{\beta\alpha}{(1+\beta)h}, \text{~~~~and~~~~} \gamma = \frac{\alpha^2}{(1+\beta)h},
	\end{equation}
	hence
	\begin{align*}
	\alpha_k = &~ \frac{(\mu-\gamma)h + \sqrt{(\mu-\gamma)^2h^2 + 4\gamma h}}{2} \\
	= &~ \frac{\beta\alpha}{2(1+\beta)} + \frac{1}{2}\sqrt{\frac{\beta^2\alpha^2}{(1+\beta)^2} + \frac{4\alpha^2}{1+\beta}}  \\
	= &~ \alpha
	\end{align*}
	Furthermore, we have
	\begin{align*}
	\overline{\gamma}_{k+1} = &~ (1-\alpha_k)\gamma_k + \alpha_k\mu = (1-\alpha)\gamma + \alpha\mu \\
	= &~ \gamma + (\mu-\gamma)\alpha = \gamma + \beta\frac{\alpha^2}{(1+\beta)h} \\
	= &~ (1+\beta)\gamma
	\end{align*}
	and $\gamma_{k+1} = \frac{1}{1+\beta}\overline{\gamma}_{k+1} = \gamma$. Since $\gamma_k=\gamma$ holds for $k=0$, by induction the proof is complete.
\end{proof}

\section{Proof of Lemma \ref{thm:estimate-sequence-construction}} \label{prf:estimate-sequence-construction}
\begin{proof}
	The proof is similar to~\citep[Lemma 2.2.2]{nesterov2004introductory} except that we introduce $\overline{\Phi}_{k+1}$ as an intermediate step in constructing $\Phi_{k+1}(x)$. In fact, to start we have $\Phi_0(x)\le (1-\lambda_0)f(x) + \lambda_0\Phi_0(x)\equiv\Phi_0(x)$. Moreover, assume (\ref{eq:weak-estimate-sequence-definition}) holds for some $k\ge 0$, i.e. $\Phi_k(x^*)-f(x^*)\le \lambda_k(\Phi_0(x^*)-f(x^*))$, then
	\begin{align*}
	\Phi_{k+1}(x^*) - f(x^*) \le &~ \overline{\Phi}_{k+1}(x^*) - f(x^*) \\
	\le &~ (1-\alpha_k)\Phi_k(x^*) + \alpha_k f(x^*) - f(x^*) \\
	= &~ (1-\alpha_k)(\Phi_k(x^*)-f(x^*)) \\
	\le &~ (1-\alpha_k)\lambda_k(\Phi_0(x^*)-f(x^*)) \\
	= &~ \lambda_{k+1}(\Phi_0(x^*)-f(x^*)),
	\end{align*}
	where the first inequality is due to our construction of $\Phi_{k+1}(x)$ in (\ref{eq:phi-less-overline-phi}), the second inequality due to strong convexity of $f$.
	By induction we have $\Phi_k(x^*)\le (1-\lambda_k)f(x^*) + \lambda_k\Phi_0(x^*)$ for all $k\ge 0$. It remains to note that condition \ref{eq:alpha-k-not-summable} ensures $\lambda_k\to 0$.
\end{proof}

\section{Proof of Lemma \ref{thm:complete-square}} \label{prf:complete-square}
\begin{proof}
	We prove this lemma by completing the square:
	\begin{align*}
	\overline{\Phi}_{k+1}(x) = &~\left(1 - \alpha_k\right)\left(\Phi_{k}^* + \frac{\gamma_k}{2}\|\Exp_{y_k}^{-1}(x) - \Exp_{y_k}^{-1}(v_k)\|^2\right) \\
	&~ + \alpha_k \left(f(y_k) + \langle \nabla f(y_k), \Exp_{y_k}^{-1}(x)\rangle + \frac{\mu}{2}\|\Exp_{y_k}^{-1}(x)\|^2\right) \\
	= &~ \frac{\overline{\gamma}_{k+1}}{2}\|\Exp_{y_k}^{-1}(x)\|^2 + \left\langle \alpha_k\nabla f(y_k) - \left(1-\alpha_k \right)\gamma_k \Exp_{y_k}^{-1}(v_k), \Exp_{y_k}^{-1}(x)\right\rangle  \\
	&~ + \left(1 - \alpha_k\right) \left(\Phi_k^* + \frac{\gamma_k}{2}\|\Exp_{y_k}^{-1}(v_k)\|^2 \right) + \alpha_k f(y_k) \\
	= &~ \frac{\overline{\gamma}_{k+1}}{2}\left\|\Exp_{y_k}^{-1}(x) - \left(\frac{(1-\alpha_k)\gamma_k}{\overline{\gamma}_{k+1}} \Exp_{y_k}^{-1}(v_k) - \frac{\alpha_k}{\overline{\gamma}_{k+1}} \nabla f(y_k) \right) \right\|^2 + \Phi_{k+1}^* \\
	= &~ \Phi_{k+1}^* +  \frac{\overline{\gamma}_{k+1}}{2}\left\|\Exp_{y_k}^{-1}(x) - \Exp_{y_k}^{-1}(v_{k+1}) \right\|^2
	\end{align*}
	where the third equality is by completing the square with respect to $\Exp_{y_k}^{-1}(x)$ and use the definition of $\Phi_{k+1}^*$ in (\ref{eq:phi-k+1-star}), the last equality is by the definition of $y_k$ in Algorithm \ref{alg:riemannian-ag}, and $\overline{\Phi}_{k+1}(x)$ is minimized if and only if $x= \Exp_{y_k}\left(\frac{(1-\alpha_k)\gamma_k}{\overline{\gamma}_{k+1}} \Exp_{y_k}^{-1}(v_k) - \frac{\alpha_k}{\overline{\gamma}_{k+1}} \nabla f(y_k)\right) = v_{k+1}$.
\end{proof}

\section{Proof of Lemma \ref{thm:x-k-bound}} \label{prf:x-k-bound}
\begin{proof}
	For $k=0$, $\Phi_k^*\ge f(x_k)$ trivially holds. Assume for iteration $k$ we have $\Phi_k^*\ge f(x_k)$, then from definition (\ref{eq:phi-k+1-star}) we have
	\begin{align*}
	\Phi_{k+1}^* \ge &~ \left(1 - \alpha_k\right) f(x_k) + \alpha_k f(y_k) - \frac{\alpha_k^2}{2\overline{\gamma}_{k+1}}\|\nabla f(y_k)\|^2 + \frac{\alpha_k(1-\alpha_k)\gamma_k}{\overline{\gamma}_{k+1}}\langle \nabla f(y_k), \Exp_{y_k}^{-1}(v_k)\rangle \\
	\ge &~ f(y_k) - \frac{\alpha_k^2}{2\overline{\gamma}_{k+1}}\|\nabla f(y_k)\|^2 + (1-\alpha_k)\left\langle \nabla f(y_k), \frac{\alpha_k\gamma_k}{\overline{\gamma}_{k+1}}\Exp_{y_k}^{-1}(v_k) + \Exp_{y_k}^{-1}(x_k)\right\rangle \\
	= &~ f(y_k) - \frac{\alpha_k^2}{2\overline{\gamma}_{k+1}}\|\nabla f(y_k)\|^2 \\
	= &~ f(y_k) - \frac{h_k}{2}\|\nabla f(y_k)\|^2,
	\end{align*}
	where the first inequality is due to $\Phi_k^*\ge f(x_k)$, the second due to $f(x_k)\ge f(y_k) + \langle\nabla f(y_k), \Exp_{y_k}^{-1}(x_k)\rangle$ by g-convexity, and the equalities follow from Algorithm \ref{alg:riemannian-ag}. On the other hand, we have the bound
	\begin{align*}
	f(x_{k+1}) \le &~ f(y_k) + \langle\nabla f(y_k), \Exp_{y_k}^{-1}(x_{k+1})\rangle + \frac{L}{2}\|\Exp_{y_k}^{-1}(x_{k+1})\|^2 \\
	= &~ f(y_k) - h_k\left(1 - \frac{L h_k}{2}\right)\|\nabla f(y_k)\|^2 \\
	\le &~ f(y_k) - \frac{h_k}{2}\|\nabla f(y_k)\|^2 \le \Phi_{k+1}^*,
	\end{align*}
	where the first inequality is by the $L$-smoothness assumption, the equality from the definition of $x_{k+1}$ in Algorithm \ref{alg:riemannian-ag} Line \ref{eq:x-k+1}, and the second inequality from the assumption that $h_k\le \frac{1}{L}$. Hence by induction, $\Phi_k^*\ge f(x_k)$ for all $k\ge 0$. 
\end{proof}

\section{Proof of Lemma \ref{thm:hyperbolic-squared-distance-distortion}}
\begin{lemma} \label{thm:large-c-hyperbolic}
	Let $a,b,c$ be the side lengths of a geodesic triangle in a hyperbolic space with constant sectional curvature $-1$, and $A$ is the angle between sides $b$ and $c$. Furthermore, assume $b\le\frac{1}{4},c\ge\frac{1}{2}$. Let $\triangle\bar{a}\bar{b}\bar{c}$ be the comparison triangle in Euclidean space, with $\bar{b}=b,\bar{c}=c,\bar{A}=A$, then
	\begin{equation}
	a^2\le (1+2b^2)\bar{a}^2
	\end{equation}
\end{lemma}
\begin{proof}
	We first apply \citep[Lemma 5]{zhang2016first} with $\kappa=-1$ to get 
	\begin{equation*}
	a^2 \le \frac{c}{\tanh(c)}b^2 + c^2 - 2bc\cos(A).
	\end{equation*}
	We also have
	\[ \bar{a}^2 = b^2 + c^2 - 2bc\cos(A). \]
	Hence we get
	\[ a^2-\bar{a}^2\le \left(\frac{c}{\tanh(c)}-1\right)b^2. \]
	It remains to note that for $b\le\frac{1}{4},c\ge\frac{1}{2}$,
	\[ 2a^2\ge 2(c-b)^2\ge 2\left(c-\frac{1}{4}\right)\ge\frac{c}{\tanh(1/2)}-1\ge\frac{c}{\tanh(c)}-1, \]
	which implies $a^2\le(1+2b^2)\bar{a}^2$.
\end{proof}

\begin{lemma} \label{thm:small-c-hyperbolic}
	Let $a,b,c$ be the side lengths of a geodesic triangle in a hyperbolic space with constant sectional curvature $-1$, and $A$ is the angle between sides $b$ and $c$. Furthermore, assume $b\le\frac{1}{4},c\le\frac{1}{2}$. Let $\triangle\bar{a}\bar{b}\bar{c}$ be the comparison triangle in Euclidean space, with $\bar{b}=b,\bar{c}=c,\bar{A}=A$, then
	\begin{equation}
	a^2\le (1+b^2)\bar{a}^2
	\end{equation}
\end{lemma}
\begin{proof}
	Recall the law of cosines in Euclidean space and hyperbolic space:
	\begin{align}
	\label{eq:euclidean_law_of_cosines} \bar{a}^2 = &~ \bar{b}^2 + \bar{c}^2 - 2\bar{b}\bar{c}\cos\bar{A}, \\
	\label{eq:hyperbolic_law_of_cosines} \cosh a = &~ \cosh b \cosh c - \sinh b \sinh c \cos A,
	\end{align}
	and the Taylor series expansion:
	\begin{align} 
		\label{eq:taylor_hyperbolic}
		\cosh x = &~ \sum_{n=0}^{\infty} \frac{1}{(2n)!}x^{2n},  & \sinh x = &~ \sum_{n=0}^{\infty} \frac{1}{(2n+1)!}x^{2n+1}.
	\end{align}
	We let $\bar{b}=b,\bar{c}=c,\bar{A}=A$, from Eq. (\ref{eq:euclidean_law_of_cosines})  we have
	\begin{equation} \label{eq:cosh_a_bar}
	\cosh\bar{a} =  \cosh\left(\sqrt{b^2 + c^2 - 2bc\cos A}\right)
	\end{equation}
	It is widely known that $\bar{a}\le a$.
	Now we use Eq. (\ref{eq:taylor_hyperbolic}) to expand the RHS of Eq. (\ref{eq:hyperbolic_law_of_cosines}) and Eq. (\ref{eq:cosh_a_bar}), and compare the  coefficients for each corresponding term $b^ic^j$ in the two series. Without loss of generality, we assume $i\ge j$; the results for condition $i<j$ can be easily obtained by the symmetry of $b,c$. We expand Eq. (\ref{eq:hyperbolic_law_of_cosines}) as
	\begin{align*}
	\cosh a = &~\left(\sum_{n=0}^{\infty} \frac{1}{(2n)!}b^{2n}\right)\left(\sum_{n=0}^{\infty} \frac{1}{(2n)!}c^{2n}\right) \\
	&~ - \left(\sum_{n=0}^{\infty} \frac{1}{(2n+1)!}b^{2n+1}\right)\left(\sum_{n=0}^{\infty} \frac{1}{(2n+1)!}c^{2n+1}\right)\cos A 
	\end{align*}
	where the coefficient $\alpha(i,j)$ of $b^ic^j$ is 
	\begin{equation}
	\alpha(i,j) = \left\{\begin{array}{rl}
	\frac{1}{(2p)!(2q)!}, & \text{if~} p,q\in\mathbb{N} \text{~and~} i=2p,j=2q, \\
	\frac{\cos A}{(2p+1)!(2q+1)!}, & \text{if~} p,q\in\mathbb{N} \text{~and~} i=2p+1,j=2q+1, \\
	0, & \text{otherwise.}
	\end{array}\right.
	\end{equation}
	Similarly, we expand Eq. (\ref{eq:cosh_a_bar}) as
	\begin{align*}
	\cosh\bar{a} = &~ \sum_{n=0}^{\infty} \frac{1}{(2n)!} \left(b^2+c^2-2bc\cos A\right)^n
	\end{align*}
	where the coefficient $\bar{\alpha}(i,j)$ of $b^ic^j$ is
	\begin{equation}
	\bar{\alpha}(i,j) = \left\{\begin{array}{rl}
	\frac{\sum_{k=0}^{q} \binom{p+q}{p-k, q-k, 2k}(2\cos A)^{2k}}{(2p+2q)!} , & \text{if~} p,q\in\mathbb{N} \text{~and~} i=2p,j=2q, \\
	\frac{\sum_{k=0}^{q} \binom{p+q+1}{p-k, q-k, 2k+1}(2\cos A)^{2k+1}}{(2p+2q+2)!} , & \text{if~} p,q\in\mathbb{N} \text{~and~} i=2p+1,j=2q+1, \\
	0, & \text{otherwise.}
	\end{array}\right.
	\end{equation}
	We hence calculate their absolute difference
	\begin{align*}
	& |\alpha(i,j) - \bar{\alpha}(i,j)|  \\ 
	= & \left\{\begin{array}{rl}
	\frac{\sum_{k=0}^{q} \binom{p+q}{p-k, q-k, 2k}2^{2k}\left(1-(\cos A)^{2k}\right)}{(2p+2q)!} , & \text{if~} p,q\in\mathbb{N} \text{~and~} i=2p,j=2q, \\
	\frac{\sum_{k=0}^{q} \binom{p+q+1}{p-k, q-k, 2k+1}2^{2k+1}\left(1-(\cos A)^{2k}\right)|\cos A|}{(2p+2q+2)!} , & \text{if~} p,q\in\mathbb{N} \text{~and~} i=2p+1,j=2q+1, \\
	0, & \text{otherwise.}
	\end{array}\right. \\
	\le & \left\{\begin{array}{rl}
	\frac{\sum_{k=0}^{q} \binom{p+q}{p-k, q-k, 2k}2^{2k}k}{(2p+2q)!} \sin^2 A, & \text{if~} p,q\in\mathbb{N} \text{~and~} i=2p,j=2q, \\
	\frac{\sum_{k=0}^{q} \binom{p+q+1}{p-k, q-k, 2k+1}2^{2k+1}k}{(2p+2q+2)!} \sin^2 A, & \text{if~} p,q\in\mathbb{N} \text{~and~} i=2p+1,j=2q+1, \\
	0, & \text{otherwise.}
	\end{array}\right. \\
	\le & \left\{\begin{array}{rl}
	\frac{q\sum_{k=0}^{q} \binom{p+q}{p-k, q-k, 2k}2^{2k}}{(2p+2q)!} \sin^2 A, & \text{if~} p,q\in\mathbb{N} \text{~and~} i=2p,j=2q, \\
	\frac{q\sum_{k=0}^{q} \binom{p+q+1}{p-k, q-k, 2k+1}2^{2k+1}}{(2p+2q+2)!} \sin^2 A, & \text{if~} p,q\in\mathbb{N} \text{~and~} i=2p+1,j=2q+1, \\
	0, & \text{otherwise.}
	\end{array}\right. \\
	\numberthis \label{eq:alpha-absolute-difference} = & \left\{\begin{array}{rl}
	\frac{q}{(2p)!(2q)!} \sin^2 A, & \text{if~} p,q\in\mathbb{N} \text{~and~} i=2p,j=2q, \\
	\frac{q}{(2p+1)!(2q+1)!} \sin^2 A, & \text{if~} p,q\in\mathbb{N} \text{~and~} i=2p+1,j=2q+1, \\
	0, & \text{otherwise.}
	\end{array}\right.
	\end{align*}
	where the two equalities are due to Lemma \ref{thm:multinomial-identities}, the first inequality due to the following fact
	\begin{align*}
		1 - (\cos A)^{2m} = &~ \left(1-(\cos A)^2\right)\left(1+(\cos A)^2+(\cos A)^4+\cdots+(\cos A)^{2(m-1)}\right) \\
	= &~ \sin^2 A \left(1+(\cos A)^2+(\cos A)^4+\cdots+(\cos A)^{2(m-1)}\right) \le m\sin^2 A
	\end{align*}
	By setting $q=0$, we see that in the Taylor series of $\cosh a - \cosh\bar{a}$, any term that does not include a factor of $c^2$ cancels out. By the symmetry of $b,c$, any term that does not include a factor of $b^2$ also cancels out. The term with the lowest order of power is thus $\frac{1}{4}b^2c^2\sin^2A$. Since we have $c\le \frac{1}{2}, b\le \frac{1}{4}$, the terms $|\alpha(i,j) - \bar{\alpha}(i,j)|b^ic^j$ must satisfy
	\begin{align*}
	\sum_{i,j}|\alpha(i,j) - \bar{\alpha}(i,j)|b^ic^j \le &~ \left(\frac{1}{4}+\sum_{\substack{i+j=2k, \\i,j\ge 2, k\ge 3}} \frac{i+j}{2(i!)(j!)}\frac{1}{2^{2k-4}}\right) b^2c^2\sin^2A \\
	\le &~ \left(\frac{1}{4} + \sum_{k\ge 3} \frac{1}{2^{2k-3}}\right) b^2c^2\sin^2A
	\le \frac{1}{2} b^2c^2\sin^2A \\
	= &~ \frac{1}{2} b^2\bar{a}^2\sin^2C
	\le \frac{1}{2}\bar{a}^2b^2
	\end{align*}
	where the first inequality follows from Eq. (\ref{eq:alpha-absolute-difference}) and is due to $\min(p,q)\le \frac{i+j}{2}$, the second inequality is due to $\sum_{\substack{i+j=2k\\i\ge 2,j\ge 2}}\frac{i+j}{(i!)(j!)} \le \frac{(2k)^2}{(k!)^2} \le 1$ for $k\ge 3$ and the last equality is due to Euclidean law of sines. We thus get
	\begin{equation}
	\cosh a - \cosh\bar{a} \le \sum_{i,j}|\alpha(i,j) - \bar{\alpha}(i,j)|b^ic^j\sin^2A 
	\le \frac{1}{2} b^2\bar{a}^2
	\end{equation}
	On the other hand, from the Taylor series of $\cosh$ we have
	\[ \cosh a - \cosh\bar{a} = \sum_{n=0}^{\infty}\frac{a^{2n}-\bar{a}^{2n}}{(2n)!}\ge\frac{1}{2}(a^2-\bar{a}^2), \]
	hence $a^2\le (1+b^2)\bar{a}^2$.
\end{proof}
\begin{lemma}[Two multinomial identities] \label{thm:multinomial-identities}
	For $p,q\in\mathbb{N}, p\ge q$, we have
	\begin{align}
	\frac{(2p+2q)!}{(2p)!(2q)!} = &~ \sum_{k=0}^{q} \binom{p+q}{p-k, q-k, 2k}2^{2k} \\
	\frac{(2p+2q+2)!}{(2p+1)!(2q+1)!} = &~ \sum_{k=0}^{q} \binom{p+q+1}{p-k, q-k, 2k+1}2^{2k+1}
	\end{align}
\end{lemma}
\begin{proof}
	We prove the identities by showing that the LHS and RHS correspond to two equivalent ways of counting the same quantity. For the first identity, consider a set of $2p+2q$ balls $b_i$ each with a unique index $i = 1,\dotsc,2p+2q$, we count how many ways we can put them into boxes $B_1$ and $B_2$, such that $B_1$ has $2p$ balls and $B_2$ has $2q$ balls. The LHS is obviously a correct count. To get the RHS, note that we can first put balls in pairs, then decide what to do with each pair. Specifically, there are $p+q$ pairs $\{b_{2i-1},b_{2i}\}$, and we can partition the counts by the number of pairs of which we put one of the two balls in $B_2$. Note that this number must be even. If there are $2k$ such pairs, which gives us $2k$ balls in $B_2$, we still need to choose $2(q-k)$ pairs of which both balls are put in $B_2$, and the left are $p-k$ pairs of which both balls are put in $B_1$. The total number of counts given $k$ is thus 
	\[ \binom{p+q}{p-k, q-k, 2k}2^{2k} \]
	because we can choose either ball in each of the $2k$ pairs leading to $2^{2k}$ possible choices. Summing over $k$ we get the RHS. Hence the LHS and the RHS equal. The second identity can be proved with essentially the same argument.
\end{proof}

%\section{Proof of Theorem \ref{thm:squared-distance-ratio-bound}}

\section{Proof of Theorem \ref{thm:convergence-induction}} \label{prf:convergence-induction}
\begin{proof} 
	\emph{The base case.} First we verify that $y_0, y_1$ is sufficiently close to $x^*$ so that the comparison inequality (\ref{eq:base-change-assumption}) holds at step $k=0$. In fact, since $y_0=x_0$ by construction, we have 
	\begin{equation} \label{eq:xstar-y0-}
	\|\Exp_{y_0}^{-1}(x^*)\| =\|\Exp_{x_0}^{-1}(x^*)\| \le  \frac{1}{4\sqrt{K}}, \qquad 5K\|\Exp_{y_0}^{-1}(x^*)\|^2 \le \frac{1}{80}\left(\frac{\mu}{L}\right)^{\frac{3}{2}} \le  \beta.
	\end{equation}
	To bound $\|\Exp_{y_1}^{-1}(x^*)\|$, observe that $y_1$ is on the geodesic between $x_1$ and $v_1$. So first we bound $\|\Exp_{x_1}^{-1}(x^*)\|$ and $\|\Exp_{v_1}^{-1}(x^*)\|$. Bound on $\|\Exp_{x_1}^{-1}(x^*)\|$ comes from strong g-convexity:
	\begin{align*}
	\|\Exp_{x_1}^{-1}(x^*)\|^2\le &~ \frac{2}{\mu}(f(x_1)-f(x^*))\le \frac{2}{\mu}(f(x_0)-f(x^*))+\frac{\gamma}{\mu}\|\Exp_{x_0}^{-1}(x^*)\|^2 \\
	\le &~ \frac{L+\gamma}{\mu}\|\Exp_{x_0}^{-1}(x^*)\|^2, 
	\end{align*}
	whereas bound on $\|\Exp_{v_1}^{-1}(x^*)\|$ utilizes the tangent space distance comparison theorem. First, from the definition of $\overline{\Phi}_1$ we have
	$$\|\Exp_{y_0}^{-1}(x^*)-\Exp_{y_0}^{-1}(v_1)\|^2 = \frac{2}{\gamma}(\overline{\Phi}_1(x^*)-\Phi_1^*)\le \frac{2}{\gamma}(\Phi_0(x^*)-f(x^*))\le \frac{L+\gamma}{\gamma}\|\Exp_{x_0}^{-1}(x^*)\|^2$$
	Then note that (\ref{eq:xstar-y0-}) implies that the assumption in Theorem \ref{thm:squared-distance-ratio-bound} is satisfied when $k=0$, thus we have
	$$\|\Exp_{v_1}^{-1}(x^*)\|^2\le  (1+\beta)\|\Exp_{y_0}^{-1}(x^*)-\Exp_{y_0}^{-1}(v_1)\|^2\le \frac{2(L+\gamma)}{\gamma}\|\Exp_{x_0}^{-1}(x^*)\|^2.$$
	Together we have 
	\begin{align}
	\nonumber\|\Exp_{y_1}^{-1}(x^*)\|\le ~& \|\Exp_{x_1}^{-1}(x^*)\| + \frac{\alpha\gamma}{\gamma+\alpha\mu}\|\Exp_{x_1}^{-1}(v_1)\|\\  \nonumber\le ~& \|\Exp_{x_1}^{-1}(x^*)\| + \frac{\alpha\gamma}{\gamma+\alpha\mu}\left(\|\Exp_{x_1}^{-1}(x^*)\| + \|\Exp_{v_1}^{-1}(x^*)\|\right) \\
	\nonumber\le ~& \sqrt{\frac{L+\gamma}{\mu}}\|\Exp_{x_0}^{-1}(x^*)\| + \frac{\alpha\gamma}{\gamma+\alpha\mu}\left(\sqrt{\frac{L+\gamma}{\mu}}+\sqrt{\frac{2(L+\gamma)}{\mu}}\right)\|\Exp_{x_0}^{-1}(x^*)\| \\
	\nonumber\le ~& \left(1 + \frac{1+\sqrt{2}}{2}\right)\sqrt{\frac{L+\gamma}{\mu}}\|\Exp_{x_0}^{-1}(x^*)\| \\
	\le ~& \frac{1}{10\sqrt{K}}\left(\frac{\mu}{L}\right)^{\frac{1}{4}} 
	\le \frac{1}{4\sqrt{K}} \label{eq:base-xstar-y-}
	\end{align}
	which also implies
	\begin{equation}
	5K\|\Exp_{y_1}^{-1}(x^*)\|^2 \le \frac{1}{20}\sqrt{\frac{\mu}{L}} \le \beta \label{eq:base-beta-}
	\end{equation}
	By (\ref{eq:base-xstar-y-}), (\ref{eq:base-beta-}) and Theorem \ref{thm:squared-distance-ratio-bound} it is hence guaranteed that 
	$$\gamma \| \Exp_{y_1}^{-1}(x^*) -\Exp_{y_1}^{-1}(v_1)\|^2 \le \overline{\gamma} \|\Exp_{y_0}^{-1}(x^*)-\Exp_{y_0}^{-1}(v_1)\|^2.$$
	\emph{The inductive step.}
	Assume that for $i=0,\dots,k-1$, (\ref{eq:base-change-assumption}) hold simultaneously, i.e.:
	$$\gamma \| \Exp_{y_{i+1}}^{-1}(x^*) -\Exp_{y_{i+1}}^{-1}(v_{i+1})\|^2 \le \overline{\gamma}\|\Exp_{y_i}^{-1}(x^*)-\Exp_{y_i}^{-1}(v_{i+1})\|^2, \forall i=0,\dots,k-1$$
	and also that $\|\Exp_{y_k}^{-1}(x^*)\|\le \frac{1}{10\sqrt{K}}\left(\frac{\mu}{L}\right)^{\frac{1}{4}}$. 
	To bound $\|\Exp_{y_{k+1}}^{-1}(x^*)\|$, observe that $y_{k+1}$ is on the geodesic between $x_{k+1}$ and $v_{k+1}$. So first we bound $\|\Exp_{x_{k+1}}^{-1}(x^*)\|$ and $\|\Exp_{v_{k+1}}^{-1}(x^*)\|$.
	Note that due to the sequential nature of the algorithm, statements about any step only depend on its previous steps, but not any step afterwards. 
	Since (\ref{eq:base-change-assumption}) hold for steps $i=0,\dots,k-1$, the analysis in the previous section already applies for steps $i=0,\dots,k-1$. Therefore by Theorem  \ref{thm:main-theorem-general-scheme} and the proof of Lemma \ref{thm:x-k-bound} we know 
	\begin{align*}
	f(x^*)\le &~ f(x_{k+1})\le\Phi_{k+1}^*\le\Phi_{k+1}(x^*)
	\le f(x^*)+(1-\alpha)^{k+1}(\Phi_0(x^*)-f(x^*)) \\
	\le &~ \Phi_0(x^*) = f(x_0)+\frac{\gamma}{2}\|\Exp_{x_0}^{-1}(x^*)\|^2
	\end{align*}
	Hence we get $f(x_{k+1})-f(x^*)\le\Phi_0(x^*)-f(x^*)$ and $\frac{\gamma}{2}\|\Exp_{y_k}^{-1}(x^*)-\Exp_{y_k}^{-1}(v_{k+1})\|^2\equiv\overline{\Phi}_{k+1}(x^*)-\Phi_{k+1}^*\le\Phi_0(x^*)-f(x^*)$. Bound on $\|\Exp_{x_{k+1}}^{-1}(x^*)\|$ comes from strong g-convexity:
	\begin{align*}
	\|\Exp_{x_{k+1}}^{-1}(x^*)\|^2\le &~ \frac{2}{\mu}(f(x_{k+1})-f(x^*))\le \frac{2}{\mu}(f(x_0)-f(x^*))+\frac{\gamma}{\mu}\|\Exp_{x_0}^{-1}(x^*)\|^2 \\
	\le &~ \frac{L+\gamma}{\mu}\|\Exp_{x_0}^{-1}(x^*)\|^2, 
	\end{align*} 
	whereas bound on $\|\Exp_{v_{k+1}}^{-1}(x^*)\|$ utilizes the tangent space distance comparison theorem. First, from the definition of $\overline{\Phi}_{k+1}$ we have
	$$\|\Exp_{y_k}^{-1}(x^*)-\Exp_{y_k}^{-1}(v_{k+1})\|^2 = \frac{2}{\gamma}(\overline{\Phi}_{k+1}(x^*)-\Phi_{k+1}^*)\le \frac{2}{\gamma}(\Phi_0(x^*)-f(x^*))\le \frac{L+\gamma}{\gamma}\|\Exp_{x_0}^{-1}(x^*)\|^2$$
	Then note that the inductive hypothesis implies that
	\begin{align*}
	\|\Exp_{v_{k+1}}^{-1}(x^*)\|^2\le (1+\beta)\|\Exp_{y_k}^{-1}(x^*)-\Exp_{y_k}^{-1}(v_{k+1})\|^2\le \frac{2(L+\gamma)}{\gamma}\|\Exp_{x_0}^{-1}(x^*)\|^2
	\end{align*}
	Together we have
	\begin{align*}
	\|\Exp_{y_{k+1}}^{-1}(x^*)\|\le ~& \|\Exp_{x_{k+1}}^{-1}(x^*)\| + \frac{\alpha\gamma}{\gamma+\alpha\mu}\|\Exp_{x_{k+1}}^{-1}(v_{k+1})\|\\  \le ~& \|\Exp_{x_{k+1}}^{-1}(x^*)\| + \frac{\alpha\gamma}{\gamma+\alpha\mu}\left(\|\Exp_{x_{k+1}}^{-1}(x^*)\| + \|\Exp_{v_{k+1}}^{-1}(x^*)\|\right) \\
	\le ~& \sqrt{\frac{L+\gamma}{\mu}}\|\Exp_{x_0}^{-1}(x^*)\| + \frac{\alpha\gamma}{\gamma+\alpha\mu}\left(\sqrt{\frac{L+\gamma}{\mu}}+\sqrt{\frac{2(L+\gamma)}{\mu}}\right)\|\Exp_{x_0}^{-1}(x^*)\| \\
	\le ~& \left(1 + \frac{1+\sqrt{2}}{2}\right)\sqrt{\frac{L+\gamma}{\mu}}\|\Exp_{x_0}^{-1}(x^*)\| \\
	\le ~& \frac{1}{10\sqrt{K}}\left(\frac{\mu}{L}\right)^{\frac{1}{4}} 
	\le \frac{1}{4\sqrt{K}}
	\end{align*}
	which also implies that
	$$ 5K\|\Exp_{y_{k+1}}^{-1}(x^*)\|^2 \le \frac{1}{20}\sqrt{\frac{\mu}{L}} \le \beta$$
	By the two lines of equations above and Theorem \ref{thm:squared-distance-ratio-bound} it is guaranteed that $\|\Exp_{y_{k+1}}^{-1}(x^*)\|\le \frac{1}{10\sqrt{K}}\left(\frac{\mu}{L}\right)^{\frac{1}{4}}$ and also
	$$\gamma \| \Exp_{y_{k+1}}^{-1}(x^*) -\Exp_{y_{k+1}}^{-1}(v_{k+1})\|^2 \le \overline{\gamma} \|\Exp_{y_k}^{-1}(x^*)-\Exp_{y_k}^{-1}(v_{k+1})\|^2.$$
	i.e. (\ref{eq:base-change-assumption}) hold for $i=0,\dots,k$. This concludes the inductive step.\\
	By induction, (\ref{eq:base-change-assumption}) hold for all $k\ge 0$, hence by Theorem \ref{thm:main-theorem-general-scheme}, Algorithm \ref{alg:constant-step} converges, with
	$$\alpha_i\equiv \alpha=\frac{\sqrt{\beta^2+4(1+\beta)\mu h}-\beta}{2} = \frac{\sqrt{\mu h}}{2}\left(\sqrt{\frac{1}{25}+4\left(1+\frac{\sqrt{\mu h}}{5}\right)} - \frac{1}{5}\right)\ge \frac{9}{10}\sqrt{\frac{\mu}{L}}.$$
\end{proof}

\end{document}